\documentclass[11pt, a4paper, reqno]{amsart}
\usepackage[T1]{fontenc}
\usepackage[latin1]{inputenc}
\usepackage[english]{babel}
\usepackage[reqno]{amsmath}
\usepackage{amsthm}
\usepackage{latexsym}
\usepackage{amssymb}
\usepackage{mathrsfs}
\usepackage{mathtools}
\usepackage{mathabx}
\usepackage{lipsum}
\usepackage{titlesec}
\usepackage{tikz-cd}
\usepackage{dsfont} 

\usepackage[totalwidth=15cm,totalheight=20cm, hmarginratio=1:1]{geometry}
\usepackage{mathtools}
\usepackage{color}
\usepackage{physics}
\usepackage{faktor}
\usepackage[shortlabels]{enumitem}
\usepackage{tikz} 
\usepackage{hyperref}
\usepackage{graphicx}
\usepackage{tabularx}
\usepackage{cases} 
\newcolumntype{C}{>{\centering\arraybackslash}X} 
\usepackage{csquotes}
\usepackage{url}

\usepackage[style=alphabetic]{biblatex}

\usepackage{accents}

\makeatletter
\newtheorem*{rep@theorem}{\rep@title}
\newcommand{\newreptheorem}[2]{%
\newenvironment{rep#1}[1]{%
 \def\rep@title{#2 \ref{##1}}%
 \begin{rep@theorem}}%
 {\end{rep@theorem}}}
\makeatother

\makeatletter
\newtheorem*{rep@cor}{\rep@title}
\newcommand{\newrepcor}[2]{%
\newenvironment{rep#1}[1]{%
 \def\rep@title{#2 \ref{##1}}%
 \begin{rep@cor}}%
 {\end{rep@cor}}}
\makeatother

\makeatletter
\newtheorem*{rep@prop}{\rep@title}
\newcommand{\newrepprop}[2]{%
\newenvironment{rep#1}[1]{%
 \def\rep@title{#2 \ref{##1}}%
 \begin{rep@prop}}%
 {\end{rep@prop}}}
\makeatother

\newtheorem{theorem}{Theorem}[section]
\newreptheorem{theorem}{Theorem}
\numberwithin{theorem}{section}

\newtheorem{manualtheoreminner}{Theorem} 
\newenvironment{manualtheorem}[1]{%
    \renewcommand\themanualtheoreminner{#1}%
    \manualtheoreminner
}{\endmanualtheoreminner}

\newtheorem{lemma}[theorem]{Lemma}

\newrepcor{corollary}{Corollary}

\newtheorem{proposition}[theorem]{Proposition}

\newrepprop{prop}{Proposition}

\theoremstyle{definition}
\newtheorem*{definition*}{Definition}
\newtheorem{definition}[theorem]{Definition}

\theoremstyle{remark}
\newtheorem{remark}[theorem]{Remark}

\makeatletter
\def\paragraph{\@startsection{paragraph}{4}%
  \z@\z@{-\fontdimen2\font}%
  {\normalfont\bfseries}}
\makeatother

\numberwithin{equation}{section}
\allowdisplaybreaks

\patchcmd{\subsection}{-.5em}{.5em}{}{}

\makeatletter

\renewcommand\section{\@startsection{section}{1}%
  \z@{.7\linespacing\@plus\linespacing}{.5\linespacing}%
  {\normalfont\scshape\centering}}

\renewcommand\subsection{\@startsection{subsection}{2}%
  \z@{-.5\linespacing\@plus-.7\linespacing}{.5\linespacing}%
  {\bfseries}}
  
\renewcommand\subsubsection{\@startsection{subsubsection}{3}%
  \z@{-.5\linespacing\@plus-.7\linespacing}{.5\linespacing}%
  {\itshape}}
  
\def\l@paragraph{\@tocline{4}{0pt}{1pc}{7pc}{}}

\makeatother

\newcommand{\C}{\mathbb{C}}
\newcommand{\R}{\mathbb{R}}

\newcommand{\Z}{\mathbb{Z}}

\newcommand{\h}{\mathbb{H}}

\newcommand{\GH}{\mathcal{GH}}

\newcommand{\Lsl}{\mathfrak{sl}}

\newcommand{\diag}{\mathrm{diag}}

\newcommand{\vl}{|}
\newcommand{\Ker}{\mathrm{Ker}}

\newcommand{\Sym}{\text{Sym}}
\newcommand{\PSL}{\mathbb{P}\mathrm{SL}}

\newcommand{\Ree}{\mathcal{R}e}

\newcommand{\Hit}{\mathrm{Hit}}

\newcommand{\g}{\mathbf{g}}

\newcommand{\Sg}{\Sigma}

\newcommand{\bigslant}[2]{{\raisebox{.2em}{$#1$}\left/\raisebox{-.2em}{$#2$}\right.}}

\DeclarePairedDelimiterX{\scal}[2]{\langle}{\rangle}{#1 \mid #2}
\DeclarePairedDelimiterX{\scall}[2]{\langle}{\rangle}{#1, #2}

\DeclareMathOperator{\Hom}{Hom}

\DeclareMathOperator{\SL}{\mathrm{SL}}

\DeclareMathOperator{\Ad}{Ad}

\newcommand{\SO}{\mathrm{SO}}
\newcommand{\charvarietyn}{\mathfrak{R}_{2,n+1}(\Sigma)}

\newcommand{\charvarietynmaximal}{\mathfrak{R}_{2,n+1}^{\text{max}}(\Sigma)}

\newcommand{\charvarietymaximal}{\mathfrak{R}^{\text{max}}_{2,3}(\Sigma)}
\newcommand{\charvarietymaximalstiefel}{\mathcal{M}(\Sigma)^{sw_1}_{sw_2}}

\newcommand{\Hyp}{\mathbb{H}^{2,n}}
\newcommand{\Hyptilde}{\widetilde{\mathbb{H}}^{2,n}}
\newcommand{\Hypdue}{\mathbb{H}^{2,2}}
\newcommand{\Hypduetilde}{\widetilde{\mathbb{H}}^{2,2}}
\newcommand{\bi}{\mathbf{b}}
\newcommand{\Higgsn}{\mathcal{H}_{2,n+1}(X)}

\newcommand{\lso}{\mathfrak{so}}

\begin{document}

\setcounter{secnumdepth}{3}

\title[Riemannian metric for representations acting on pseudo-hyperbolic space]{Riemannian geometry of maximal surface group representations acting on pseudo-hyperbolic space}

\author[Nicholas Rungi]{Nicholas Rungi}
\address{NR: Universit\'e Grenoble Alpes, Institut Fourier, Grenoble, France.} \email{nicholas.rungi@univ-grenoble-alpes.fr} 

\date{\today}
\begin{abstract}For any maximal surface group representation into $\SO_0(2,n+1)$, we introduce a non-degenerate scalar product on the the first cohomology group of the surface with values in the associated flat bundle. In particular, it gives rise to a non-degenerate Riemannian metric on the smooth locus of the subset consisting of maximal representations inside the character variety. In the case $n=2$, we carefully study the properties of the Riemannian metric on the maximal connected components, proving that it is compatible with the orbifold structure and finding some totally geodesic sub-varieties. Then, in the general case, we explain when a representation with Zariski closure contained in $\SO_0(2,3)$ represents a smooth or orbifold point in the maximal $\SO_0(2,n+1)$-character variety and we show that the associated space is totally geodesic for any $n\ge 3$.\end{abstract}
\maketitle
\tableofcontents

\section{Introduction}
Let $\Sg$ be a closed, connected and oriented surface of genus $g\ge 2$, with universal cover $\widetilde\Sg$. In recent years, people have been interested in studying geometric and dynamical properties of surface group representations into real semisimple Lie group $G$ of rank greater than one, with the aim of generalizing Teichm\"uller theory which concerns the group $\PSL(2,\R)$ (\cite{Wienhard_intro}). For a large class of Lie groups, the space formed by discrete and faithful representations consists of a union of connected components in the character variety. Such a subset is called a \emph{higher Teichm\"uller space}, and depending on the property of the Lie group, it contains an isomorphic copy of Teichm\"uller space $\mathcal{T}(\Sg)$, referred to as the \emph{Fuchsian locus}, and denoted with $\mathcal{F}(\Sg)$. \\ \\ From a differential geometric point of view, on Teichm\"uller space it was possible to define Riemannian metrics (\cite{weil1958modules,ahlfors1961curvature,ahlfors1961some,wolpert1986thurston}) and asymmetric distances (\cite{thurston1998minimal,papadopoulos2015finsler}), many of which are natural in terms of hyperbolic structures on the surface. A fundamental result of Bers (\cite{bers1961correction},\cite{bers1965automorphic}) also proved the existence of a natural complex structure on $\mathcal{T}(\Sg)$ which turns out to be compatible with the Weil-Petersson Riemannian metric (\cite{weil1958modules}), written in terms of quadratic differentials on the associated Riemann surface, and giving rise to a K\"ahler metric on Teichm\"uller space. Nowadays, many properties of the Weil-Petersson metric are known, such as: holomorphic sectional, scalar and Ricci curvature being strictly negative, incompleteness, and invariance for the action of the mapping class group of the surface.  \\ \\ It is then natural to ask whether the same differential geometric structures can also be defined on higher Teichm\"uller spaces, so that they restrict on the Fuchsian locus to those mentioned above. In the last twenty years, remarkable results have been obtained in this regard; in fact, it has been possible to define (pseudo)-Riemannian metrics (\cite{bridgeman2015pressure,li2013teichm,kim2017kahler,tamburelli2021riemannian,mazzoli2021parahyperkahler,rungi2021pseudo,rungi2023mathbb,rungi2023moduli}), asymmetric distances (\cite{carvajales2022thurston}), symplectic forms (\cite{goldman1984symplectic}), and complex structures (\cite{loftin2001affine,Labourie_cubic,labourie2017cyclic,alessandrini2019geometry}), trying to generalize those known on $\mathcal{T}(\Sg)$.    \\ \\ In this paper, we focus on a higher Teichm\"uller space associated with the Lie group $G=\SO_0(2,n+1)$ defined as the identity component of the group of linear transformations of $\R^{n+3}$ preserving a bi-linear form of signature $(2,n+1)$, for $n\ge 2$. Since $\SO_0(2,n+1)$ is a rank $2$ Lie group of Hermitian type, it makes sense to consider the Toledo invariant $\tau(\rho)\in\Z$ associated with a representation $\rho:\pi_1(\Sg)\to\SO_0(2,n+1)$, which will be called \emph{maximal} if $|\tau(\rho)|$ is equal to $-\chi(\Sg)$ (\cite{burger2010surface}). The space $\charvarietynmaximal$ of all maximal surface group representations into $\SO_0(2,n+1)$ (up to conjugation) has a structure of a real analytic (possibly singular) variety and it is a union of connected components (\cite{bradlow2006maximal}). Recently, Collier-Tholozan-Toulisse have found a nice geometric property for such representations, namely they proved that for any $\rho\in\charvarietynmaximal$ there exists a unique $\rho$-equivariant maximal space-like embedding $\varphi:\widetilde\Sg\to\Hyp$ (\cite{collier2019geometry}). Making use of this result, for any maximal representation $\rho$, we construct a \emph{geometric} Riemannian metric on the Zariski tangent space of $\charvarietynmaximal$ at $\rho$, if the representation is a smooth point of the variety. We follow an analogous construction performed for $\SL(3,\R)$ (\cite{li2013teichm}) using hyperbolic affine spheres in $\R^3$, and for $\SO_0(2,2)$ (\cite{tamburelli2021riemannian}) using the embedding of space-like surfaces in globally hyperbolic maximal compact Anti-de Sitter three-manifolds (GHMC for short).
\begin{manualtheorem}A \label{thm:A}
For any maximal representation $\rho:\pi_1(\Sg)\to\SO_0(2,n+1)$ which is also a smooth point in $\charvarietynmaximal$ and for any $n\ge 2$, there exists a scalar product $\g_\rho$ on the Zariski tangent space depending on the unique $\rho$-equivariant maximal space-like embedding $\varphi:\widetilde\Sg\to\Hyp$. In particular, the tensor $\g$ defines a Riemannian metric in the smooth locus of $\charvarietynmaximal$.\end{manualtheorem}
The space $\charvarietynmaximal$ is intimately related to the Higgs bundles theory through the non-abelian Hodge correspondence (\cite{garcia2009hitchin}). In this framework, it has been proven the existence of a decomposition into connected components (\cite{bradlow2006maximal})  $$\charvarietynmaximal=\bigsqcup_{\substack{sw_1\in H^1(\Sg,\Z_2) \\ sw_2\in H^2(\Sg,\Z_2)}}\charvarietynmaximal_{sw_2}^{sw_1} \ ,$$ where $sw_1\in H^1(\Sg,\Z_2)$ and $sw_2\in H^2(\Sg,\Z_2)$ are the topological invariants of the associated Higgs bundle. However, it must be emphasized that such decomposition occurs only when $n>2$, indeed in the case $n=2$, the space $\mathfrak{R}_{2,3}^{\text{max}}(\Sg)^{sw_1}_{sw_2}$ with $sw_1=0$, is no longer connected, but decomposes further into $4g-3$ connected components. Therefore, considering the space of maximal representations for $n=2$ one obtains that 
$$\charvarietymaximal=\bigg(\bigsqcup_{sw_1\neq 0, \ sw_2}\charvarietymaximalstiefel\bigg)\sqcup\bigg(\bigsqcup_{0\le d\le 4g-4}\mathcal{M}_d(\Sg)\bigg) \ ,$$ where we denoted with $\charvarietymaximalstiefel$ the connected components corresponding to $sw_1\neq 0$. It turns out that the spaces labelled by $d\in [0,4g-4]$ are smooth manifolds if $d\neq 0$, hence they carry a well-defined $\g$ according to the above theorem. On the other hand, $\charvarietymaximalstiefel$ and $\mathcal{M}_0(\Sg)$ are singular spaces. Nevertheless, using the explicit expression of the metric $\g$ and the classification of singularities (\cite{alessandrini2019geometry}) we prove the following \begin{manualtheorem}B \label{thm:B}The Riemannian metric $\g$ on $\charvarietymaximalstiefel$ and on $\mathcal{M}_0(\Sg)$ is compatible with all orbifold singularities.\end{manualtheorem}
It is interesting that singularities of the above components correspond to representations whose Zariski closure is contained in a tightly embedded subgroup $G<\SO_0(2,3)$ (\cite{alessandrini2019geometry}). For instance, when $G=\SO_0(2,2)$ we find holonomies of GHMC Anti-de Sitter three-manifolds, whose associated deformation space $\GH(\Sg)$ admits a Riemannian metric $\g_\mathrm T$ defined by Tamburelli (\cite{tamburelli2021riemannian}), and which coincides with $\g|_{\GH(\Sg)}$. Instead, when $G=\SO_0(2,1)\cong\PSL(2,\R)$ we find holonomies of hyperbolic structures, namely points in the Fuchsian locus $\mathcal{F}(\Sg)\cong\mathcal T(\Sg)$. In this regard, with a direct computation we can show that $\g|_{\mathcal F(\Sg)}$ is a constant multiple of the Weil-Petersson metric.\begin{manualtheorem}C \label{thm:C}The space $\big(\GH(\Sg),\g_\mathrm T\big)$ is totally geodesic in $\big(\charvarietymaximalstiefel,\g\big)$ and in $\big(\mathcal{M}_0(\Sg),\g\big)$. Moreover, the Fuchsian locus $\mathcal{F}(\Sg)$ is totally geodesic with respect to the Weil-Petersson metric inside $(\charvarietymaximalstiefel, \g)$.\end{manualtheorem} Regrettably, although we know that $\g$ restricts to (a multiple of) the Weil-Petersson metric in the copy of Teichm\"uller space inside $\mathcal{M}_0(\Sg)$, we cannot conclude that it is totally geodesic since it represents a non-orbifold singularity. As for the connected components $\mathcal{M}_d(\Sg)$, with $d\in(0,4g-4)$, they are the analogous counterpart of Gothen components for $\mathbb P\mathrm{Sp}(4,\R)\cong\SO_0(2,3)$ (\cite{gothen2001components}) and any representation $\rho\in\mathcal{M}_d(\Sg)$ have Zariski dense image, hence there is no subspace to look for. Nevertheless, we explain how they embed and if they represent a smooth or orbifold point in the maximal $\SO_0(2,n+1)$-character variety.  \begin{manualtheorem}D \label{thm:D}The Gothen component $\big(\mathcal{M}_d(\Sg),\g\big)$ is totally geodesic in $\big(\mathfrak{R}^{\text{max}}_{2,n+1}(\Sg)^{sw_1=0}_{sw_2=0},\g\big)$ when $d$ is even and is totally geodesic in $\big(\mathfrak{R}^{\text{max}}_{2,n+1}(\Sg)^{sw_1=0}_{sw_2\neq0},\g\big)$ when $d$ is odd, for $n\ge 3$.\end{manualtheorem}
Similarly, it makes sense to ask a similar question for the spaces $\charvarietymaximalstiefel$ seen as sub-varieties of $\charvarietynmaximal^{sw_1\neq 0}_{sw_2}$. In particular, we obtain the following:
\begin{manualtheorem}E \label{thm:E} 
    The connected components $\big(\charvarietymaximalstiefel,\g\big)$ are totally geodesic sub-varieties of $\big(\charvarietynmaximal^{sw_1\neq 0}_{sw_2},\g\big)$
\end{manualtheorem}
The last result we obtained concerns the $\SO_0(2,3)$-Hitchin component $\Hit(\Sg)$, which corresponds to the connected component $\mathcal{M}_d(\Sg)$ with $d=4g-4$, hence it carries a well-defined $\g$ being smooth. The statement is quite surprising in that it shows a significant difference with the case of $\SL(3,\R)$ (\cite{li2013teichm}). In fact, in our case, the Fuchsian locus embeds as a totally geodesic submanifold in the Hitchin component but, with an explicit computation, we show that $\g$ does not restrict to the Weil-Petersson metric. In this regard, it is worth mentioning two other constructions on the $\SO_0(2,3)$-Hitchin component: the Labourie-Wentworth pressure metric (\cite{labourie2018variations} see also \cite{bridgeman2015pressure}) and the Carvajales-Dai-Pozzetti-Wienhard asymmetric distance (\cite{carvajales2022thurston}), both defined on much more general spaces. Regarding the former, it has been shown that it restricts to the Weil-Petersson metric on the Fuchsian locus, whereas the latter restricts to the so called asymmetric Thurston metric on Teichm\"uller space (\cite{thurston1998minimal}), which is not Riemannian. In agreement with the explicit expression of our metric $\g$ on the Fuchsian locus, developed in Section \ref{sec:4.5}, it can be concluded that $\g$ potentially defines a different Riemannian metric from the one given by Labourie-Wentworth and the one by Carvajales-Dai-Pozzetti-Wienhard being asymmetric.  \\ \\ 
To our knowledge, $\g$ is the first Riemannian metric on $\charvarietymaximalstiefel$ and on $\mathcal{M}_0^{\text{max}}(\Sg)$ that is shown to be compatible with its orbifold structure. It would be interesting to understand its relation to the natural complex structure defined by Alessandrini-Collier (\cite{alessandrini2019geometry}) and to the Goldman symplectic form (\cite{goldman1984symplectic}). The same question applies to the Hitchin (\cite{labourie2017cyclic}) and Gothen (\cite{alessandrini2019geometry}) components.
\subsubsection*{Outline of the paper}In Section \ref{sec:2} we recall the basic concepts that we will need later, including pseudo-hyperbolic space, maximal surfaces in $\Hyp$, the connected component decomposition of $\charvarietynmaximal$ and the role of Higgs bundles. In Sections \ref{sec:3.1} and \ref{sec:3.2} we explain how to construct the Riemannian metric $\g$ on the smooth locus of $\charvarietynmaximal$, using the result on equivariant maximal surfaces in $\Hyp$, from which Theorem \ref{thm:A} will be deduced. In Section \ref{sec:4.2} we recall the classification of singularities in $\charvarietymaximalstiefel$ and $\mathcal{M}_0(\Sg)$ and we prove Theorem \ref{thm:B}. Then, in Section \ref{sec:4.3} we explain how to obtain Theorem \ref{thm:C}, and we explicitly show that $\g$ restricts to a multiple of the Weil-Petersson metric. Moving on, in Section \ref{sec:5.1} we study representations in $\charvarietynmaximal^{sw_1\neq 0}_{sw_2}$ factoring through the sub-variety $\charvarietymaximalstiefel$ and we prove Theorem \ref{thm:E}. In particular, it is needed an analysis of the type of singularities that these representations form. Finally, in Section \ref{sec:4.4} we study the properties of the metric on Gothen components by proving Theorem \ref{thm:D}.
\subsection*{Acknowledgements}
The author is greatful to Andrea Tamburelli for his constant support during the preparation of the paper, and for reading a first draft of this work. The author would also like to thank Brian Collier for useful discussions on the topic that have improved some statements of the main theorems, and the anonymous referee for some comments that enhanced the exposition of the paper.
\section{Background materials}\label{sec:2}
In this section we first recall the definition of pseudo-hyperbolic space and maximal space-like surfaces (\cite{ishihara1988maximal},\cite{collier2019geometry},\cite{labourie2020plateau},\cite{nie2022cyclic},\cite{labourie2023quasicircles}). Then, we introduce maximal surface group representations into $\SO_0(2,n+1)$ and we briefly explain their relation to $\SO_0(2,n+1)$-maximal Higgs bundles through the non-abelian Hodge correspondence (\cite{collier2019geometry}).
\subsection{Pseudo-hyperbolic space and maximal space-like surfaces}\label{sec:2.1}Let $n$ be a non negative integer and let us denote with $\R^{2,n+1}$ the space $\R^{n+3}$ endowed with the following symmetric non-degenerate bi-linear form: $$\langle x,y\rangle:=x_1y_1+x_2y_2-x_3y_3\dots-x_{n+3}y_{n+3},\quad x,y\in\R^{n+3} \ .$$ It is clear from the definition that $\langle\cdot,\cdot\rangle$ has signature equal to $(2,n+1)$. Let us denote with $\bi$ the associated quadratic form and consider the subspace of $\R^{2,n+1}$ consisting of all vectors with norm equal to $-1$, namely $$\Hyptilde:=\{ x\in \R^{2,n+1} \ | \ \bi(x) =-1 \} \ .$$ It is classically known that the tangent space at a point $x\in\Hyptilde$ can be identified with the orthogonal complement $\{\R\cdot x\}^\perp\subset\R^{2,n+1}$ so that the restriction of the indefinite bi-linear form $\langle\cdot,\cdot\rangle$ on such a subspace induces a pseudo-Riemannian metric $g_{\Hyp}$ of signature $(2,n)$, and of constant sectional curvature $-1$.
\begin{definition}The \emph{pseudo-hyperbolic space} of signature $(2,n)$ is defined as $$\Hyp:=\bigslant{\Hyptilde}{\{\pm\mathrm{Id}\}} \ .$$
\end{definition} Given that the elements $\pm\mathrm{Id}$ act on $\Hyptilde$ by isometries, there is a pseudo-Riemannian metric, still denoted with $g_{\Hyp}$, on $\Hyp$ of signature $(2,n)$ and of constant sectional curvature $-1$, induced by the natural quotient projection from $\Hyptilde$ to $\Hyp$, which is a covering of degree $2$. \\ \\ Now let $S$ be a connected smooth surface without boundary, we say that $\varphi:S\rightarrow\Hyp$ is a \emph{space-like} embedding if $\varphi$ is an embedding and the induced metric $g_T:=\varphi^*g_{\Hyp}|_{TS}$ is Riemannian. In particular, the embedded surface is called \emph{complete} if $g_T$ is a complete Riemannian metric. The \emph{normal bundle} $NS$ of the embedding is defined as the $g_{\Hyp}$-orthogonal of $TS$ inside $T\Hyp|_S$ and it inherits a metric $g_N:=\varphi^*g_{\Hyp}|_{NS}$. The pull-back of the Levi-Civita connection $\nabla$ of $g_{\Hyp}$ by $\varphi$ decomposes, according to the splitting $T\Hyp|_S=TS\oplus NS$, as $$\nabla=\begin{pmatrix} \nabla^T & -B \\ \Pi & \nabla^N\end{pmatrix} \ .$$ In the above matrix representation, $\nabla^T$ is the Levi-Civita connection of $g_T$ and $\nabla^N$ is a connection on $NS$ preserving $g_N$ and called the \emph{normal connection}. The tensor $\Pi$ is an element of $\Omega^1(S,\Hom(TS,NS))$ called the \emph{second fundamental form}, while the tensor $B$ is an element of $\Omega^1(S,\Hom(NS,TS))$ called the \emph{shape operator}. They are related by the following equation
\begin{equation}g_N\big(\Pi(X,Y),\xi\big)=g_T\big(Y,B(X,\xi)\big) , \quad X,Y\in\Gamma(TS), \ \xi\in\Gamma(NS) \ .\end{equation}
Moreover, the second fundamental form is \emph{symmetric}, namely $\Pi(X,Y)=\Pi(Y,X)$ for any $X,Y\in\Gamma(TS)$, hence can be seen as an element of $\mathrm{Sym}^2\big(T^*S\big)\otimes NS$.\begin{definition}Let $S\subset\Hyp$ be a space-like surface and let $\{e_1,e_2\}$ be a $g_T$-orthonormal basis of $TS$, then $S$ is called \emph{maximal} if \begin{equation}\label{maximal}\tr_{g_T}\Pi:=\Pi(e_1,e_1)+\Pi(e_2,e_2)=0 \ .\end{equation}\end{definition}

\subsection{The character variety and maximal surface group representations}\label{sec:2.2}
Here we first recall the definition of the Lie group $\SO_0(2,n+1)$, then we introduce its associated character variety and the notion of maximal surface group representations. The material presented here is already known in the literature and we will recall only what is necessary for the purposes of the article. From now on and throughout the rest of the paper, we will denote by $\Sg$ a closed smooth oriented surface of genus $g\ge 2$ and by $\widetilde\Sg$ its universal cover. \\ \\
The Lie group $\SO_0(2,n+1)$ is the identity component of the group of linear transformations of $\R^{n+3}$ preserving $\bi$, which acts in a natural and transitive way on $\Hyp$. In other words, it is the identity component of $$\mathrm{SO}(2,n+1)=\{Q\in\SL(n+3,\R ) \ | \ Q^T\mathrm{I}_{2,n+1}Q=\mathrm{I}_{2,n+1}\} \ ,$$ where $\mathrm{I}_{2,n+1}:=\text{diag}(1,1,-1,\dots,-1)$ is the matrix associated with $\bi$ in an orthonormal basis of $\R^{n+3}$. Now consider the space Hom$(\pi_1(\Sg),\SO_0(2,n+1))$ of all representations from $\pi_1(\Sg)$ to $\SO_0(2,n+1)$. This set has a topology induced by the inclusion $$\mathrm{Hom}(\pi_1(\Sg),\SO_0(2,n+1))\hookrightarrow\SO_0(2,n+1)^{2g}$$ which sends the representation $\rho$ to the t-uple $\big(\rho(a_1),\dots,\rho(b_g)\big)$, where $a_1,\dots,b_g$ are generators of $\pi_1(\Sg)$ subject to the relation $\prod_{i=1}^g\big[a_i,b_i\big]=1$. There is a natural action of $\SO_0(2,n+1)$ on this space given by conjugation: for $\gamma\in\pi_1(\Sg)$ and $Q\in\SO_0(2,n+1)$ $$(Q\cdot\rho)(\gamma):=Q^{-1}\rho(\gamma)Q \ . $$
In order to get a Hausdorff quotient space, one needs to restrict to the \emph{completely reducible} representations, i.e. those $\rho:\pi_1(\Sg)\to\SO_0(2,n+1)$ which split as a direct sum of irreducible representations. Let us denote with $\Hom^+\big(\pi_1(\Sg),\SO_0(2,n+1)\big)$ the space of such representations, then the $\SO_0(2,n+1)$-\emph{character variety} is defined as $$\charvarietyn:=\bigslant{\Hom^+\big(\pi_1(\Sg),\SO_0(2,n+1)\big)}{\SO_0(2,n+1)} \ .$$ The topological space just defined has a structure of real algebraic variety, possibly singular. The Zariski tangent space of $\charvarietyn$ has a nice description at smooth points, which now we are going to recall. Let $\rho:\pi_1(\Sg)\longrightarrow\SO_0(2,n+1)$ be a representation and let us consider its adjoint representation $\Ad\rho$ into $\lso_0(2,n+1)$, namely the Lie algebra of $\SO_0(2,n+1)$. One can define a flat $\lso_0(2,n+1)$-bundle over $\Sg$, with holonomy given by $\Ad\rho$, as follows: $$\lso_0(2,n+1)_{\Ad\rho}:=\bigslant{\big(\widetilde\Sg\times\lso_0(2,n+1)\big)}{\sim} \ ,$$ where $(\widetilde x,v)\sim (\gamma\cdot\widetilde x, \Ad\rho(\gamma)v)$ for any $\gamma\in\pi_1(\Sg)$ and $\widetilde x\in\widetilde\Sg, v\in\lso_0(2,n+1)$. In particular, it makes sense to consider the cohomology group of the surface with values in the flat bundle $\lso_0(2,n+1)_{\Ad\rho}$, and the following classical result is obtained: \begin{theorem}[\cite{goldman1984symplectic}]\label{thm:tangenttocharactervariety}
If $\rho\in\charvarietyn$ is a smooth point of the character variety, then the Zariski tangent space at $\rho$ can be identified with $H^1(\Sg,\lso_0(2,n+1)_{\Ad\rho})$. 
\end{theorem}
\begin{remark}
It is worth mentioning that the theorem just stated applies more generally to Lie groups that admit a non-degenerate, symmetric and Ad-invariant bi-linear form on their Lie algebra. Moreover, in the original paper by Goldman (\cite{goldman1984symplectic}) the statement is given in terms of the first group cohomology $H^1\big(\pi_1(\Sg), \lso_0(2,n+1)\big)$, which is in fact isomorphic to $H^1(\Sg,\lso_0(2,n+1)_{\Ad\rho})$ when $\Sg$ is a closed surface of genus $g\ge 2$.
\end{remark}Now, let us shift our attention to a subspace of the representation space that we will study in the following sections. Firstly, let us mention that to any representation $\rho:\pi_1(\Sg)\to\SO_0(2,n+1)$ it is possible to associate an $\mathrm{SO}(2)$-bundle $E_\rho$ over $\Sg$ whose only topological invariant is the Euler class (\cite{collier2019geometry}). Then, the \emph{Toledo invariant} $\tau(\rho)$ is defined as the Euler class of such $\mathrm{SO}(2)$-bundle $E_\rho\to\Sg$ (\cite{burger2010surface},\cite{collier2019geometry}). In particular, we have a map $$\tau:\Hom\big(\pi_1(\Sg),\SO_0(2,n+1)\big)\to\Z$$ which is known to be continuous, hence locally constant, and invariant by conjugation. Therefore, it induces a map at the level of the character variety, still denoted with $\tau$ by abuse of notation. Moreover, for any representation $\rho:\pi_1(\Sg)\to\SO_0(2,n+1)$, the Toledo invariant $\tau$ satisfies the following \emph{Milnor-Wood} inequality (\cite{domic1987gromov}): $$\vl\tau(\rho)\vl\le 2g-2 \ .$$
\begin{definition}\label{def:maximalrep}
A representation $\rho:\pi_1(\Sg)\to\SO_0(2,n+1)$ is \emph{maximal} if $\vl\tau(\rho)\vl=2g-2$.
\end{definition}
The Toledo number is defined in terms of the Euler class of the $\mathrm{SO}(2)$-bundle $E_\rho\to\Sg$, hence it depends on the orientation on the surface. In particular, by taking the opposite orientation, it changes the sign of the Toledo invariant. For this reason, from now one, we will assume that $\tau(\rho)\ge 0$, for any representation $\rho:\pi_1(\Sg)\to\SO_0(2,n+1)$, so that $\rho$ is maximal if $\tau(\rho)=2g-2$. We will denote with $\charvarietynmaximal$ the subspace of the character variety consisting only of maximal representations. It is well known that $\charvarietynmaximal$ is a union of connected components which we are going to describe through the nonabelian Hodge correspondence in the next section.
\subsection{Nonabelian Hodge correspondence}\label{sec:2.3}
Here we introduce the notion of $\SO_0(2,n+1)$-Higgs bundles over a fixed Riemann surface, and we explain how their moduli space is related to the character variety through non-abelian Hodge correspondence (\cite{hitchin1987self,donaldson1987twisted,corlette1988flat,simpson1992higgs,garcia2009hitchin}). There are many references in the literature on this topic, even in more general contexts, for this reason here we will introduce only the necessary notions for our purposes.
\\ \\ Let $X$ be a fixed Riemann surface structure on $\Sg$. We will denote by $K_X$ its canonical bundle, namely the holomorphic cotangent bundle of $X$, and by $\mathcal{O}_X$ the trivial holomorphic bundle. The notion of Higgs bundle was first introduced by Hitchin in the $\SL(2,\C)$ case (\cite{hitchin1987self}), and then generalized by Simpson (\cite{simpson1988constructing}) for a semi-simple complex lie group. 
\begin{definition}
An $\SL(r,\C)$-Higgs bundle on $X$ is a pair $(\mathcal E,\Phi)$, where $\mathcal E$ is a rank $r$ holomorphic vector bundle over $X$ such that $\bigwedge^r\mathcal E\cong\mathcal O_X$, and $\Phi$ is a holomorphic section of $\mathrm{End}(\mathcal{E})\otimes K_X$ such that $\tr(\Phi)=0$.
\end{definition}
To any such pair $(\mathcal E,\Phi)$ one can associate an integer number $d$ called the \emph{degree}. It can be defined as the integral of the first Chern class of $\mathcal{E}$ over $X$, and it can be shown that it coincides with the degree of the associated line bundle $\bigwedge^r\mathcal{E}$. Thus, it is clear that $\SL(r,\C)$-Higgs bundles have $d=0$. Such an integer number serves to introduce an algebraic notion of stability, which we now recall. 
\begin{definition}\label{def:stabilityHiggs}
Let $(\mathcal E,\Phi)$ be an $\SL(r,\C)$-Higgs bundle, then: \begin{itemize}
    \item[$\bullet$] $(\mathcal E,\Phi)$ is \emph{stable} if for any proper subbundle $0\neq\mathcal F\subset\mathcal{E}$ with $\Phi(\mathcal{F})\subset\mathcal F\otimes K_X$, one has $\deg(\mathcal{F})<0$;
    \item[$\bullet$] $(\mathcal E,\Phi)$ is \emph{polystable} if it is a direct sum of $\SL(r_i,\C)$-Higgs bundles $(\mathcal E_i,\Phi_i)$ such that $\deg(\mathcal{E}_i)=0$ for any $i$.
\end{itemize} \end{definition}
In the most general case, the definition of Higgs bundle and stability can also be given for a real reductive Lie group $G$ (\cite{garcia2009hitchin}). When $G=\SO_0(2,n+1)$ (\cite[\S 2.3]{collier2019geometry}), we obtain the following
\begin{definition}
An $\SO_0(2,n+1)$-Higgs bundles over a Riemann surface $X$ is a quint-uple $(\mathcal L,\mathcal V,b_\mathcal V,\beta,\gamma)$ such that: \begin{itemize}
    \item[$\bullet$] $\mathcal V$ is a rank $n+1$ holomorphic bundle over $X$ with a trivialization $\bigwedge^{n+1}\mathcal V\cong\mathcal O_X$, and $b_\mathcal V$ is a non-degenerate holomorphic section of $\Sym^2(\mathcal V^*)$; \item[$\bullet$] $\mathcal L\to X$ is a holomorphic line bundle; \item[$\bullet$] $\gamma\in H^0\big(X,\mathcal L^{-1}\otimes\mathcal V\otimes K_X\big)$ and $\beta\in H^0\big(X, \mathcal L\otimes\mathcal V\otimes K_X\big)$.  
\end{itemize} 
\end{definition}It is interesting to note that, given an $\SO_0(2,n+1)$-Higgs bundle $(\mathcal L,\mathcal V,b_\mathcal V,\beta,\gamma)$ we can construct an $\SL(n+3,\C)$-Higgs bundles over $X$ as follows: $\mathcal E:=\mathcal L\oplus\mathcal V\oplus\mathcal L^{-1}$ and $$\Phi=\begin{pmatrix}
    0 & \beta^\dag & 0 \\ \gamma & 0 & \beta \\ 0 & \gamma^\dag & 0
\end{pmatrix}:\mathcal E\longrightarrow\mathcal E\otimes K_X \ , $$ where $\beta^\dag:=\beta^T\circ b_\mathcal V:\mathcal V\to\mathcal L\otimes K_X$ and $\gamma^\dag:=\gamma^T\circ b_\mathcal V:\mathcal V\to\mathcal L^{-1}\otimes K_X$ are holomorphic sections. In particular, we will say that an $\SO_0(2,n+1)$-Higgs bundle is polystable if and only if the associated $\SL(n+3,\C)$-Higgs bundle is polystable according to Definition \ref{def:stabilityHiggs}. Polystability for a Higgs bundle is equivalent to the existence of a Hermitian metric, compatible with the holomorphic structure, that satisfies some gauge theoretic equations, known in the literature as \emph{self-duality equations}. 
In particular, the Hermitian connection on the polystable $\SO_0(2,n+1)$-Higgs bundle induces a flat connection on the associated $\SL(n+3,\C)$-Higgs bundle, whose holonomy gives a completely reducible representation $\rho:\pi_1(\Sg)\to\SO_0(2,n+1)$.
\begin{proposition}[\cite{collier2019geometry}]\label{prop:polystableHiggs}
If $(\mathcal L,\mathcal V,b_\mathcal V,\beta,\gamma)$ is polystable, then $\deg\mathcal L=\tau(\rho)$, where $\rho:\pi_1(\Sg)\to\SO_0(2,n+1)$ is the associated completely reducible representation and $\tau(\rho)$ is its Toledo invariant. In particular, if $\deg\mathcal L=2g-2$, then: \begin{itemize}
    \item[$\bullet$] there is a holomorphic $b_\mathcal V$-orthogonal decomposition $\mathcal V=\mathcal I\oplus\mathcal W$, with $\mathcal W$ a rank $n$ holomorphic bundle over $X$, $\mathcal I\cong\bigwedge^n\mathcal W$ and $\mathcal I^2\cong\mathcal O_X$; \item[$\bullet$] $\mathcal L\cong\mathcal I\otimes K_X$;
    \item[$\bullet$] $\gamma\cong\begin{pmatrix}
        1 \\ 0
    \end{pmatrix}$ and $\beta=\begin{pmatrix}
        q_2 \\ \beta_0
    \end{pmatrix}:\mathcal K_X^{-1}\otimes\mathcal I\to\big(\mathcal I\otimes K_X\big)\oplus\big(\mathcal W\otimes K_X\big)$, with $q_2\in H^0\big(X,K_X^2\big)$ and $\beta_0\in H^0\big(X,\mathcal W\otimes\mathcal I\otimes K_X^2\big)$.
\end{itemize}
\end{proposition} The Higgs bundles of Proposition \ref{prop:polystableHiggs} with $\deg\mathcal L=2g-2$ are called \emph{maximal}, and they are determined by a quadruple $(\mathcal W,b_\mathcal W,q_2,\beta_0)$, where $b_\mathcal W$ is the restriction of $b_\mathcal V$ to $\mathcal W$. The main result needed for our purposes and concerning these newly defined objects is the following:
\begin{proposition}[\cite{bradlow2006maximal}]
The two topological invariants of a polystable maximal $\SO_0(2,n+1)$-Higgs bundle $(\mathcal W,b_\mathcal W,q_2,\beta_0)$ are the first and second Stiefel-Whitney class $sw_1\in H^1(\Sg,\Z_2), sw_2\in H^2(\Sg,\Z_2)$ of $\mathcal W$.
\end{proposition}
In general, one can define the so-called \emph{gauge group} $\mathcal G$ acting on the space of polystable $G$-Higgs bundles, with $G$ a real reductive group (\cite[\S 2.2]{alessandrini2019geometry}). In our case, this allows us to define a moduli space for such objects as: $$\Higgsn:=\bigslant{\{\text{polystable} \ \SO_0(2,n+1)\text{-Higgs bundles over} \ X\}}{\mathcal G} \ .$$
\begin{theorem}[Nonabelian Hodge correspondence]\label{thm:nonabelianHodge}
    Let $\Sigma$ be a smooth closed oriented surface of genus $g\ge 2$, then for each choice of a Riemann surface structure $X$ on $\Sg$, there is a real analytic isomorphism between the moduli space $\Higgsn$ of polystable $\SO_0(2,n+1)$-Higgs bundles on $X$ and the $\SO_0(2,n+1)$-character variety $\charvarietyn$. 
\end{theorem}Although we only gave the definition of the moduli space for $\SO_0(2,n+1)$, it must be pointed out that the above result has been proven for a general real reductive Lie group $G$ (\cite{garcia2009hitchin}). In particular, one can go back and forth between results regarding representations and Higgs bundles. For this reason, using the topological invariants of polystable $\SO_0(2,n+1)$-Higgs bundles, we get a connected components decomposition for the maximal representation space \begin{theorem}[\cite{bradlow2006maximal}]\label{thm:decompositioncomponents} 
For any $n>2$ the characteristic classes $sw_1\in H^1(\Sg,\Z_2)$ and $sw_2\in H^2(\Sg,\Z_2)$ distinguish connected components in the moduli space of maximal polystable $\SO_0(2,n+1)$-Higgs bundles. In particular, they induce a decomposition for the space $\charvarietynmaximal$ as follows: $$\bigsqcup_{\substack{sw_1\in H^1(\Sg,\Z_2) \\ sw_2\in H^2\Sg,\Z_2)}}\charvarietynmaximal_{sw_2}^{sw_1} \ ,$$ where each $\charvarietynmaximal_{sw_2}^{sw_1}$ denotes the set of maximal representations such that the Stiefel-Whitney classes of $\mathcal W$ are $sw_1$ and $sw_2$.
\end{theorem}
Each space in the decomposition of Theorem \ref{thm:decompositioncomponents} is non-empty and connected, and there are a total of $2^{2g+1}$ components. In the case of $\SO_0(2,3)$ the decomposition is slightly different and will be addressed in Section \ref{sec:4.1}.
\section{A Riemannian metric for $\SO_0(2,n+1)$ maximal representations}
In Section \ref{sec:3.1}, for any maximal representation into $\SO_0(2,n+1)$, we present a construction of a non-degenerate scalar product on the first cohomology group of the surface with values in the associate flat bundle. Then, in Section \ref{sec:3.2}, we show that such a pairing is geometric, in the sense that its definition relies on the theory of maximal surfaces in $\Hyp$. 
\subsection{Definition of the metric}\label{sec:3.1}
According to Theorem \ref{thm:tangenttocharactervariety}, if $\rho\in\charvarietyn$ is a smooth point then it is sufficient to define a metric on $H^1\big(\Sg,\lso_0(2,n+1)_{\Ad\rho}\big)$. In order to do this, it is sufficient to have chosen a scalar product $\iota$ on $\lso_0(2,n+1)_{\Ad\rho}$ and a metric $h$ on $\Sg$. In fact, let us assume for a moment that this is the case, then a Riemannian metric on the first cohomology group follows from a variant of Hodge theory that we now recall (see \cite{raghunathan1972discrete}). The metric $h$ and the orientation on $\Sg$ allow us to define a scalar product $(\cdot,\cdot)_h$ on the space of smooth $k$-forms $\Omega^k(\Sg)$, and an Hodge-star operator $$\ast:\Omega^k(\Sg)\longrightarrow\Omega^{2-k}(\Sg) \ ,$$ by imposing $\alpha\wedge(\ast\beta)=(\alpha,\beta)_h\mathrm d\mathrm{Vol}_h$, where $\mathrm d\mathrm{Vol}_h$ denotes the area form on $\Sg$ induced by $h$. In addition, the choice of the scalar product $\iota$ enables us to define a bi-linear pairing $g$ on the space of $\lso_0(2,n+1)_{\Ad\rho}$-valued $1$-forms as follows: \begin{equation}\label{eq:definitiong}
    g\big(\sigma\otimes\phi, \sigma'\otimes\phi'\big):=\int_\Sg\iota(\phi,\phi')\sigma\wedge(\ast\sigma') \ , 
\end{equation}where $\sigma,\sigma'\in\Omega^1(\Sg)$ are $1$-forms on the surface and $\phi,\phi'\in\Gamma(\lso_0(2,n+1)_{\Ad\rho})$ are smooth sections of the flat bundle.\newline Now, given $\rho\in\charvarietyn$ one can consider the associated \emph{contragradient} representation $\rho^*:\pi_1(\Sg)\to\SO_0(2,n+1)$ defined as $$\big(\rho^*(\gamma)\cdot T\big)(v):=T\big(\rho(\gamma)^{-1}\cdot v\big) \ ,$$ for every $v\in\R^{n+3}$ and $T\in\big(\R^{n+3}\big)^*\equiv\Hom(\R^{n+3},\R)$. In particular, it is possible to consider the flat bundle $\lso_0(2,n+1)_{\Ad\rho^*}$ whose fibre turns out to be the dual fibre of $\lso_0(2,n+1)_{\Ad\rho}$. The inner product $\iota$ induces a bundle isomorphism $$\#:\lso_0(2,n+1)_{\Ad\rho}\longrightarrow\lso_0(2,n+1)_{\Ad\rho^*}$$ defined by $(\#A)(B):=\iota(A,B)$ for $A,B\in\lso_0(2,n+1)_{\Ad\rho}$. This can be naturally extended to an isomorphism at the level of bundle-valued $k$-forms on the surface $$\#:\Omega^k\big(\Sg, \lso_0(2,n+1)_{\Ad\rho}\big)\longrightarrow\Omega^k\big(\Sg, \lso_0(2,n+1)_{\Ad\rho^*}\big) \ .$$ By considering the induced exterior derivative $\mathrm d$ and the induced Hodge-star operator $\ast$ on $\Omega^k\big(\Sg, \lso_0(2,n+1)_{\Ad\rho}\big)$, we can introduce a co-boundary operator $$\delta:\Omega^k\big(\Sg, \lso_0(2,n+1)_{\Ad\rho}\big)\longrightarrow\Omega^{k-1}\big(\Sg, \lso_0(2,n+1)_{\Ad\rho}\big)$$ by setting $\delta:=-(\#)^{-1}\ast^{-1}\mathrm d\ast\#$, and a Laplacian operator $$\Delta:\Omega^k\big(\Sg, \lso_0(2,n+1)_{\Ad\rho}\big)\longrightarrow\Omega^k\big(\Sg, \lso_0(2,n+1)_{\Ad\rho}\big) \ ,$$ by setting $\Delta:=\delta\mathrm d+\mathrm d\delta$. According to the above construction, an $\lso_0(2,n+1)_{\Ad\rho}$-valued $1$-form $\eta$ is said to be $\Delta$-harmonic (\emph{harmonic} for short) if $\Delta\eta=0$, which is equivalent to $d\eta=\delta\eta=0$. As for the classical Hodge theory, there is an orthogonal decomposition: $$\Omega^1\big(\Sg, \lso_0(2,n+1)_{\Ad\rho}\big)=\Ker(\Delta)\oplus\Im(\delta)\oplus\Im(d)$$ and every cohomology class contains a unique harmonic representative. In other words, there is an isomorphism $H^1\big(\Sg,\lso_0(2,n+1)_{\Ad\rho}\big)\cong\mathrm{Ker}(\Delta)$. Thus, the bi-linear pairing $g$ induces a scalar product in cohomology as follows: \begin{equation}\label{eq:definitiongoncohomology}\g\big([\alpha],[\beta]\big):=g(\alpha_{\text{harm}},\beta_{\text{harm}}) \ ,\qquad [\alpha],[\beta]\in H^1\big(\Sg,\lso_0(2,n+1)_{\Ad\rho}\big),\end{equation} where $\alpha_{\text{harm}}$ and $\beta_{\text{harm}}$ are the harmonic representatives of $\alpha$ and $\beta$.
\subsection{Relation with maximal space-like surfaces in \texorpdfstring{$\Hyp$}{H^{2,n}}}\label{sec:3.2}
Recall that if $\rho\in\charvarietynmaximal$ is a smooth point, then the Zariski tangent space $T_{[\rho]}\charvarietynmaximal$ is identified with $H^1\big(\Sg,\lso_0(2,3)_{\Ad\rho}\big)$ (see Theorem \ref{thm:tangenttocharactervariety}). In Section \ref{sec:3.1} we explained how a Riemannian metric can be defined on the above first cohomology group, depending on the choice of a metric $h$ on $\Sg$ and a scalar product $\iota$ on $\lso_0(2,n+1)_{\text{Ad}\rho}$. Here we want to show that, whenever the representation $\rho$ is maximal, there is a natural geometric choice for $h$ and $\iota$. In order to do this we need to recall the following crucial result: \begin{theorem}[\cite{collier2019geometry}]\label{thm:rhoequivariantsurface}
If $\rho:\pi_1(\Sg)\to\SO_0(2,n+1)$ is a maximal representation, then there exists a unique $\rho$-equivariant maximal space-like embedding $\varphi:\widetilde\Sg\to\Hyp$.
\end{theorem} 
For any maximal representation $\rho:\pi_1(\Sg)\to\SO_0(2,n+1)$ the unique maximal space-like embedding $\varphi:\widetilde\Sg\to\Hyp$ induces a metric $g_T$ on $\widetilde\Sg$ as explained in Section \ref{sec:2.1}. Moreover, by $\rho$-equivariance we get an induced Riemannian metric $h$ on $\Sg\cong\widetilde\Sg/\rho\big(\pi_1(\Sg)\big)$. As for $\iota$, however, we first need to introduce a scalar product on $\R^{n+3}$ which is related with the maximal surface and $h$. In fact, for any $\widetilde x\in\widetilde\Sg$ we have a frame of $\R^{n+3}$ formed by the unit tangent vectors $u_1(\widetilde x)$ and $u_2(\widetilde x)$ to the surface at $\varphi(\widetilde x)$, the unit time-like normal vectors $N_1(\widetilde x),\dots,N_n(\widetilde x)$ at $\varphi(\widetilde x)$ and the position vector $\varphi(\widetilde x)$. It is possible to define a scalar product $\iota_{\widetilde x}$ on $\R^{n+3}$, depending on the point $\widetilde x\in\widetilde \Sg$, by declaring the frame $\{u_1(\widetilde x),u_2(\widetilde x),\varphi(\widetilde x),N_1(\widetilde x),\dots,N_n(\widetilde x)\}$ to be orthonormal. In particular, since $\lso_0(2,n+1)\subset\mathfrak{gl}(n+3,\R)\cong\R^{n+3}\times\big(\R^{n+3}\big)^*$, we can restrict the scalar product $\iota_{\widetilde x}\otimes\iota^*_{\widetilde x}$ on $\lso_0(2,n+1)$, hence it can be induced on the trivial bundle $\widetilde\Sg\times\lso_0(2,n+1)$. This descend to a metric $\iota$ on $\lso_0(2,n+1)_{\Ad\rho}$ by setting: \begin{equation}\label{eq:iota}\iota_p(\phi,\phi'):=\iota_{\widetilde x}(\widetilde\phi_{\widetilde x},\widetilde\phi_{\widetilde x}') \ \ \text{for some} \ \widetilde x\in\pi^{-1}(p),\end{equation} where $p\in\Sg$, $\pi:\widetilde\Sg\to\Sg$ is the universal cover projection and $\widetilde\phi,\widetilde\phi'$ are lifts of $\phi,\phi'$ to the trivial bundle $\widetilde\Sg\times\lso_0(2,n+1)$ evaluated at $\widetilde x$. Since the maximal space-like embedding is $\rho$-equivariant, so is the scalar product $\iota_{\widetilde x}$. Moreover, with the same approach as for $\Lsl(3,\R)$ (\cite{li2013teichm}), it is easy to check that $\iota_p$ does not depend on the choice of the point in the fibre $\pi^{-1}(p)$, so that $\iota$ gives rise to a well-defined metric on the flat bundle $\lso_0(2,n+1)_{\Ad\rho}\to\Sg$. \begin{theorem}\label{thm:metricsmoothpoints}
For any maximal representation $\rho:\pi_1(\Sg)\to\SO_0(2,n+1)$ which is also a smooth point in $\charvarietynmaximal$, there exists a scalar product $\g_\rho$ on $H^1(\Sg,\lso_0(2,n+1)_{\Ad\rho})$ depending on the induced metric $h$ on $\Sg\cong\widetilde\Sg/\rho\big(\pi_1(\Sg)\big)$ and the inner product $\iota$ on $\lso_0(2,3)_{\Ad\rho}$. In particular, the tensor $\g$ defines a Riemannian metric in the smooth locus of $\charvarietynmaximal$.
\end{theorem}
\begin{remark}\label{rem:metricSOLiTamburelli}
A similar result was obtained on the Hitchin component for $\SL(3,\R)$ (\cite{li2013teichm}) and on the maximal component of the character variety for $\SO_0(2,2)\cong\PSL(2,\R)\times\PSL(2,\R)$ (\cite{tamburelli2021riemannian}). In both cases, the components are smooth manifolds and the construction above can be performed at each point. In our case, as we will see in detail for $n=2$, the space of maximal representations is not in general a smooth manifold but has orbifold and non-orbifold singularities. Nevertheless, we will show that, at least for the orbifolds ones, the Riemannian metric $\g$ is compatible with such singularities.
\end{remark}
We conclude the section by stating a technical result that will be useful for some forthcoming computations \begin{lemma}[\cite{li2013teichm}]\label{lem:matrixscalarproduct}
Let $\widetilde x$ be any point in the fibre $\pi^{-1}(p)$ for some $p\in\Sg$, and suppose we have a matrix representation $H$ of the inner product $\iota_{\widetilde x}$ with respect to the canonical basis of $\R^{n+3}$. Then, $$\iota_p(M,N)=\tr(M^tH^{-1}NH), \quad \text{for} \ M,N\in\lso_0(2,n+1) \ .$$ \end{lemma}
\section{Sub-varieties for $n=2$}
Here we specialize the discussion to maximal representations into $\SO_0(2,3)$ whose associated space will be denoted with $\charvarietymaximal$. In Section \ref{sec:4.1} we recall the connected components decomposition, and in Section \ref{sec:4.2} we show that our metric is compatible with the orbifold structure of $\charvarietymaximal$. In Section \ref{sec:4.3} we study the restriction of the metric $\g$ on some interesting subspaces of the character variety and we explain how it is related to the metric defined by Tamburelli for $\SO_0(2,2)$ maximal representations (\cite{tamburelli2021riemannian}). Then, we prove that, for some connected components, $\g$ restricts on the Fuchsian locus to a multiple of the Weil-Petersson metric on Teichm\"uller space, which embeds as a totally geodesic sub-variety. 
Finally, in Section \ref{sec:4.5} there is the most surprising but also mysterious part of the construction: the metric $\g$ on the $\SO_0(2,3)$-Hitchin component does not restrict to the Weil-Petersson metric on the copy of Teichm\"uller space, which is in contrast with the $\SL(3,\R)$ case (\cite{li2013teichm}).
\subsection{Maximal connected components}\label{sec:4.1}
The decomposition in connected components we presented in Theorem \ref{thm:decompositioncomponents} holds only for $n>2$. This is because the case $n=2$ is somehow special due to the presence of additional connected components. Recall that a maximal polystable $\SO_0(2,3)$-Higgs bundle is a quadr-uple $(\mathcal W,b_\mathcal W,q_2,\beta_0)$ as in Proposition \ref{prop:polystableHiggs}, where $\mathcal W$ is a holomorphic rank $2$ vector bundle over the Riemann surface $X:=(\Sg,J)$. In particular, when the first Stiefel-Whitney class of $\mathcal W$ vanishes, it is endowed with an $\SO(2,\C)$-structure (\cite{alessandrini2019geometry}). Thus, there is a further holomorphic splitting $$(\mathcal W,b_\mathcal W)=\bigg(\mathcal F\oplus\mathcal F^{-1},\begin{pmatrix}
    0 & 1 \\ 1 & 0 
\end{pmatrix}\bigg) \ ,$$ where $\mathcal F$ is a holomorphic line bundle over $X$ (\cite{mumford1971theta}). \begin{theorem}[\cite{bradlow2006maximal}]\label{thm:decompositionduetre}
In the notation above, the space of maximal surface group representations into $\SO_0(2,3)$ decomposes as: \begin{equation}\label{eq:decompositionSO23}\charvarietymaximal=\bigg(\bigsqcup_{sw_1\neq 0, \ sw_2}\charvarietymaximalstiefel\bigg)\sqcup\bigg(\bigsqcup_{0\le d\le 4g-4}\mathcal{M}_d(\Sg)\bigg) \ ,\end{equation}where $sw_1\in H^1(\Sg,\Z_2)$ and $sw_2\in H^2(\Sg,\Z_2)$ represent, respectively, the first and second Stiefel-Whitney class of $\mathcal W$, and $d$ is the degree of the holomorphic line bundle $\mathcal F$.
\end{theorem}
It is interesting to note that for representations whose corresponding $\SO_0(2,3)$-Higgs bundle has vanishing first Stiefel-Whitney class, there is a further decomposition dictated by the integer number $d$. In fact, for $d=4g-4$ we retrieve the Hitchin component, denoted with $\Hit(\Sg)$, for $d=0$ we have the most singular component denoted with $\mathcal{M}_0(\Sg)$ and the remaining $4g-5$ connected components are the equivalent of the Gothen components (\cite{gothen2001components}) under the isomorphism $\mathbb P\mathrm{Sp}(4,\R)\cong\SO_0(2,3)$. All spaces with $d\in(0,4g-4]$ are smooth manifolds, hence the construction of Theorem \ref{thm:metricsmoothpoints} applies. As for the components $\charvarietymaximalstiefel$ they have at most orbifold singularities.  \begin{theorem}\label{thm:metricGothencomponentHitchin}
There is a well-defined Riemannian metric, still denoted with $\g$, on the $\SO_0(2,3)$-Hitchin component and all Gothen components \end{theorem}
\subsection{Orbifold singularities}\label{sec:4.2}
In this section, we show that the Riemannian metric $\g$ is compatible with the orbifold structure of $\charvarietymaximal$. This will be accomplished by looking at how maximal representations into $\SO_0(2,3)$ can factorize through some subgroups. Let us briefly recall that whenever we have a reductive subgroup $G<\SO_0(2,3)$, the Zariski closure of a completely reducible representation $\rho:\pi_1(\Sg)\to\SO_0(2,3)$ is contained in $G$, up to conjugation, if and only if the corresponding polystable $\SO_0(2,3)$-Higgs bundle reduces to a $G$-Higgs bundle (\cite{alessandrini2019geometry}). Furthermore, it is shown (\cite{burger2010surface}) that if the Zariski closure of a maximal representation $\rho:\pi_1(\Sg)\to\SO_0(2,3)$ is contained in a proper subgroup $G$, then $G$ is of Hermitian type and the inclusion map $G\to\SO_0(2,3)$ is a \emph{tight embedding} (\cite{hamlet2014tight}), namely the property of being maximal is preserved. In particular, we have the following list of tightly embedded Lie subgroups of $\SO_0(2,3)$ (\cite{hamlet2014tight}): \begin{itemize}
    \item[$\bullet$] $\SO_0(2,1)$ with the inclusion induced by the $5$-dimensional irreducible representation;
    \item[$\bullet$] $\SO_0(2,1)\times\SO(2)$ with the inclusion induced by the isometric embedding $\R^{2,1}\to\R^{2,3}$ which sends $(x_1,x_2,x_3)\mapsto (x_1,x_2,x_3,0,0)$;
    \item[$\bullet$] $\SO_0(2,2)\times\SO(1)$ with the inclusion induced by the isometric embedding $\R^{2,2}\to\R^{2,3}$ which sends $(x_1,x_2,x_3,x_4)\mapsto (x_1,x_2,x_3,x_4,0)$.
\end{itemize}
The maximal representations factoring through the first group in the above list form the \emph{Fuchsian locus} $\mathcal F(\Sg)$ in the Hitchin component, which is an isomorphic copy of Teichm\"uller space $\mathcal T(\Sg)$ of the surface. Each representation in $\mathcal F(\Sg)$ can be written as $j\circ\rho_{\text{Fuch}}$, where $\rho_{\text{Fuch}}$ is a Fuchsian representation into $\SO_0(2,1)$ and $j$ is the unique irreducible representation of $\SO_0(2,1)$ into $\SO_0(2,3)$. The representations that factorize through the second Lie group of the list form the Fuchsian locus in the connected components $\charvarietymaximalstiefel$ and $\mathcal{M}_0(\Sg)$. Each such representation can be written as $(\rho_{\text{Fuch}}\otimes\det\alpha)\oplus\alpha$, with $\alpha:\pi_1(\Sg)\to\mathrm O(2)$, and Teichm\"uller space is found by taking $\alpha$ the trivial representation. Finally, those factoring through $\SO_0(2,2)$ can be seen as holonomies of globally hyperbolic maximal compact (GHMC for short) anti-de Sitter $3$-manifolds isomorphic to $\Sg\times\R$ (\cite{mess2007lorentz}) and they are contained in both $\charvarietymaximalstiefel$ and $\mathcal{M}_0(\Sg)$. It is also worth mention that, any representation in $\mathcal{M}_d(\Sg)$, with $d\in(0,4g-4)$ is Zariski dense (\cite{alessandrini2019geometry}), hence there is no subspace to look for. \begin{proposition}[\cite{alessandrini2019geometry}]\label{prop:orbifoldHiggs}
All singularities of $\charvarietymaximalstiefel$ are orbifold type. Moreover, a maximal representation $\rho\in\charvarietymaximalstiefel$ defines a  \begin{itemize}
    \item[$\bullet$] $\Z_2$-orbifold point if its Zariski closure is contained in a tightly embedded copy of $\SO_0(2,2)\times\SO(1)$ or $\SO_0(2,1)\times\SO(2)$; \item[$\bullet$] $\Z_2\times\Z_2$-orbifold point if its Zariski closure is contained a tightly embedded copy of $\SO_0(2,1)\times\SO(1)\times\SO(1)$; \item[$\bullet$] smooth point otherwise. 
\end{itemize}
\end{proposition}
A clarification needs to be given regarding the last result. Proposition \ref{prop:orbifoldHiggs} was proven in its version for $\SO_0(2,3)$-Higgs bundles, which, thanks to nonabelian Hodge correspondence and the discussion at the beginning of the section, can be translated in the context of maximal representations. In particular, in our framework, orbifold points are generated by the centralizer $\mathcal C_\rho:=C\big(\rho(\pi_1(\Sg))\big)<\SO_0(2,3)$, which can be isomorphic to $\Z_2$ or $\Z_2\times\Z_2$. Such a centralizer acts by conjugacy on the cohomology $H^1\big(\Sg,\lso_0(2,3)_{\Ad\rho}\big)$, which as we explained (see Theorem \ref{thm:tangenttocharactervariety}), is identified with the tangent to the character variety if the representation is smooth. Thus, in order to show that our metric $\g$ is compatible with the orbifold structure we firstly need to show that the metric $g$ on the space of $\lso_0(2,3)_{\Ad\rho}$-valued $1$-forms (see Section \ref{sec:3.1}) is invariant for the centralizer action, and secondly that the induced map on cohomology preserves the harmonicity of $1$-forms. These two facts put together will give a well-defined Riemannian metric on: $$\bigslant{H^1(\Sg,\lso_0(2,3)_{\Ad\rho})}{\mathcal C_\rho}\cong T_{[\rho]}\charvarietymaximal \ , $$ whenever $\rho$ is an orbifold point of $\charvarietymaximal$.  \begin{lemma}\label{lem:centralizerorbifoldpoint}
Let $\rho:\pi_1(\Sg)\to\SO_0(2,3)$ be a maximal representation in $\charvarietymaximalstiefel$. Then \begin{itemize}
    \item[$\bullet$] if $\rho$ is a $\Z_2$-orbifold point with Zariski closure contained in $\SO_0(2,2)\times\SO(1)$ its centralizer is generated by $A=\diag(-1,-1,-1,-1,1)$, otherwise if its Zariski closure is contained in $\SO_0(2,1)\times\SO(2)$ the centralizer is generated by $B=\diag(1,1,1,-1,-1)$;
    \item[$\bullet$] if $\rho$ is a $\Z_2\times\Z_2$-orbifold point its centralizer is given by $\{\mathrm{Id_5},A,B,C\}$ where $A$ and $B$ are the same as above, $\mathrm{Id}_5$ is the $5\times 5$ identity matrix and $C:=A\cdot B=\diag(-1,-1,-1,1,-1)$.
\end{itemize}
\end{lemma}
\begin{proof}
Let $\rho$ be any maximal representation in $\charvarietymaximalstiefel$. Suppose first that $\overline{\rho\big(\pi_1(\Sg)\big)}<\SO_0(2,2)\times\SO(1)$, then by appealing to Proposition \ref{prop:orbifoldHiggs} and Theorem \ref{thm:nonabelianHodge} we know that $\mathcal C_\rho\cong\Z_2$. The only possibility for a matrix in $A\in\mathrm \SO_0(2,3)$ to satisfy $$A\begin{pmatrix}M & 0 \\ 0 & 1
\end{pmatrix}A^{-1}=\begin{pmatrix}M & 0 \\ 0 & 1
\end{pmatrix},\quad \text{for all} \ M\in\SO_0(2,2)$$ is for it to be diagonal and with only $\pm 1$. In particular, having to preserve the top $4\times4$ block, and having to belong to $\SO_0(2,3)$ the only possibility is for it to be exactly $A=\diag(-1,-1,-1,-1,1)$, with $A^2=\mathrm{Id}_5$. In the other case, if $\overline{\rho\big(\pi_1(\Sg)\big)}<\SO_0(2,1)\times\SO(2)$, hence $\mathcal C_\rho\cong\Z_2$, the only possibility, according to the previous argument, is $B=\diag(1,1,1,-1,-1)$. Finally, if $\rho$ is a $\Z_2\times\Z_2$-orbifold point we have $\overline{\rho\big(\pi_1(\Sg)\big)}<\SO_0(2,1)\times\SO(1)\times\SO(1)$. In particular, both the above matrices $A,B\in\SO_0(2,3)$ still preserve matrices of the form $$\begin{pmatrix}
    \SO_0(2,1) & 0  \\ 0 & \mathrm{Id}_2
\end{pmatrix} \ .$$ Thus, $A$ and $B$ belong to $\mathcal C_\rho$ and the third non-trivial element is given by $C:=A\cdot B=\diag(-1,-1,-1,1,-1)$.
\end{proof}
\begin{lemma}\label{lem:invariantmetricg}
For any $\sigma,\sigma'\in\Omega^1(\Sg)$ and for any sections $\phi,\phi'$ of $\lso_0(2,3)_{\Ad\rho}$, we have $$g\big(\sigma\otimes L\phi L^{-1},\sigma'\otimes L\phi'L^{-1}\big)=g\big(\sigma\otimes\phi,\sigma'\otimes\phi'\big) \ ,$$ where $L$ is one among the non-trivial matrices that are part of the centralizer $C\big(\rho(\pi_1(\Sg))\big)$.
\end{lemma}
\begin{proof}
This is simply an application of the uniqueness statement in Theorem \ref{thm:rhoequivariantsurface}. In fact, since any of the matrix $A,B$ and $C$ of Lemma \ref{lem:centralizerorbifoldpoint} belongs to $\SO_0(2,3)$, the maximal space-like surface in $\Hypdue$ associated with $L\rho L^{-1}$ (for $L=A,B,C$), is the same as $\rho$. As a consequence, the construction of the metric $h$ on $\Sg$ and the scalar product $\iota$ on $\R^5$ (see Section \ref{sec:3.2}) is invariant by the centralizer action. Recalling that the metric $g$ on the space of $\lso_0(2,3)_{\Ad\rho}$-valued $1$-forms is given by $$g\big(\sigma\otimes\phi,\sigma'\otimes\phi'\big)=\int_\Sg\iota(\phi,\phi')\sigma\wedge(\ast_h\sigma') \ , $$ we obtain the claim.
\end{proof}
\begin{theorem}\label{thm:metriccompatiblewithorbifold}
If $\rho$ is an orbifold point in $\charvarietymaximalstiefel$, then the action of the centralizer sends harmonic forms to harmonic forms. In particular, it preserves the Riemannian metric $\g$.
\end{theorem}
\begin{proof}Let $\alpha$ be a $\lso_0(2,3)_{\Ad\rho}$-valued $1$-form, then $\alpha$ is harmonic if and only if $\mathrm d\alpha=\delta\alpha=0$ (see Section \ref{sec:3.1}). Thus, if $\sum_i\sigma_i\otimes\phi_i$ is a harmonic representative in its cohomology class we know that $\mathrm d\big(\sum_i\sigma_i\otimes\phi_i\big)=0$ and $\delta\big(\sum_i\sigma_i\otimes\phi_i\big)=0$. We need to show that these imply $\mathrm d\big(\sum_i\sigma_i\otimes L\phi_iL^{-1}\big)=0$ and $\delta\big(\sum_i\sigma_i\otimes L\phi_iL^{-1}\big)=0$, for $L=A,B,C\in\SO_0(2,3)$ generators of the centralizers $\mathcal{C}_\rho$, according to the cases of Proposition \ref{prop:orbifoldHiggs} and Lemma \ref{lem:centralizerorbifoldpoint}. Notice that the condition $\mathrm d\big(\sum_i\sigma_i\otimes L\phi_iL^{-1}\big)=0$ simply follows by linearity of $\mathrm d$. As for the $\delta$-closedness, we first need to recall that $\delta\big(\sum_i\sigma_i\otimes L\phi_iL^{-1}\big)=0$ if and only if $$\mathrm d\ast\#\Big(\sum_i\sigma_i\otimes L\phi_iL^{-1}\Big)=\mathrm d\ast\Big(\sum_i\sigma_i\otimes\# L\phi_iL^{-1}\Big)=0 \ .$$ It must be pointed out that even though the element $\big(\sum_i\sigma_i\otimes L\phi_iL^{-1}\big)$ is a $\lso_0(2,3)_{\Ad(L\rho L^{-1})}$-valued $1$-form, since $L$ always belongs to $\SO_0(2,3)$ and because of the statement of uniqueness of the maximal surface in $\Hypdue$ (Theorem \ref{thm:rhoequivariantsurface}), the operator $\#$ is the same even after applying the centralizer action, namely the action of conjugation by $L$. In order to compute $\#$ we choose the basis for $\lso_0(2,3)$ given by \begin{align*}
     E_1=\begin{psmallmatrix}
        0 & -1 & 1 & 0 & 0 \\ 1 & 0 & 0 & 0 & 0 \\ 1 & 0 & 0 & 0 & 0 \\ 0 & 0 & 0 & 0 & 0 \\ 0 & 0 & 0 & 0 & 0 
    \end{psmallmatrix},\quad E_2=\begin{psmallmatrix}
        0 & 0 & 0 & 0 & 0 \\ 0 & 0 & 1 & 0 & 0 \\ 0 & 1 & 0 & 0 & 0 \\ 0 & 0 & 0 & 0 & 0 \\ 0 & 0 & 0 & 0 & 0
    \end{psmallmatrix},\quad E_3=\begin{psmallmatrix}
        0 & 1 & 1 & 0 & 0 \\ -1 & 0 & 0 & 0 & 0 \\ 1 & 0 & 0 & 0 & 0 \\ 0 & 0 & 0 & 0 & 0 \\ 0 & 0 & 0 & 0 & 0
    \end{psmallmatrix},\quad E_4=\begin{psmallmatrix}
        0 & 0 & 0 & 1 & 0 \\ 0 & 0 & 0 & 0 & 0 \\ 0 & 0 & 0 & 0 & 0 \\ 1 & 0 & 0 & 0 & 0 \\ 0 & 0 & 0 & 0 & 0
    \end{psmallmatrix}, \end{align*} \begin{equation*}E_5=\begin{psmallmatrix}
        0 & 0 & 0 & 0 & 1 \\ 0 & 0 & 0 & 0 & 0 \\ 0 & 0 & 0 & 0 & 0 \\ 1 & 0 & 0 & 0 & 0 \\ 1 & 0 & 0 & 0 & 0
    \end{psmallmatrix}, \quad E_6=\begin{psmallmatrix}
        0 & 0 & 0 & 0 & 0 \\ 0 & 0 & 0 & 1 & 0 \\ 0 & 0 & 0 & 0 & 0 \\ 0 & 1 & 0 & 0 & 0 \\ 0 & 0 & 0 & 0 & 0
    \end{psmallmatrix},\quad E_7=\begin{psmallmatrix}
        0 & 0 & 0 & 0 & 0 \\ 0 & 0 & 0 & 0 & 1 \\ 0 & 0 & 0 & 0 & 0 \\ 0 & 0 & 0 & 0 & 0 \\ 0 & 1 & 0 & 0 & 0
    \end{psmallmatrix}, \quad E_8=\begin{psmallmatrix}
        0 & 0 & 0 & 0 & 0 \\ 0 & 0 & 0 & 0 & 0 \\ 0 & 0 & 0 & 1 & 0 \\ 0 & 0 & 1 & 0 & 0 \\ 0 & 0 & 0 & 0 & 0
    \end{psmallmatrix},
\end{equation*}\begin{equation*}
   E_9=\begin{psmallmatrix}
       0 & 0 & 0 & 0 & 0 \\ 0 & 0 & 0 & 0 & 0 \\ 0 & 0 & 0 & 0 & 1 \\ 0 & 0 & 0 & 0 & 0 \\ 0 & 0 & 1 & 0 & 0
   \end{psmallmatrix} ,\quad E_{10}=\begin{psmallmatrix}
       0 & 0 & 0 & 0 & 0 \\ 0 & 0 & 0 & 0 & 0 \\ 0 & 0 & 0 & 0 & 0 \\ 0 & 0 & 0 & 0 & 1 \\ 0 & 0 & 0 & -1 & 0
   \end{psmallmatrix} \ .
\end{equation*}
Let us denote with $\{E_j^*\}_{j=1}^{10}$ the dual basis, then by definition of $\#$ we have $$(\# M)(N)=\iota(M,N)$$ for $M,N\in\lso_0(2,3)$. Thus, $$\#M=\sum_{j=1}^{10}\iota(M,E_j)E^*_j \ , \ \text{with} \ E^*_j(E_i)=\begin{cases}
    1 \quad \text{if} \ i=j \\ 0 \quad \text{otherwise}
\end{cases} \ .$$
By hypothesis, we know that $\mathrm d\ast\big(\sum_i\sigma_i\otimes\#\phi_i\big)=0$ if and only if $$\mathrm d\ast\bigg(\sum_i\sigma_i\otimes\sum_{j=1}^{10}\iota(\phi_i,E_j)E_j^*\bigg)=0 \ ,$$which implies that \begin{equation}\label{eq:coefficientidual}
    \mathrm d\ast\bigg(\sum_i\sigma_i\iota(\phi_i,E_j)\bigg)=0 \ , \ \text{for any} \ j=1,\dots,10 \ .
\end{equation}Therefore, \begin{align*}
    \mathrm d\ast\bigg(\sum_i\sigma_i\otimes\# L\phi_iL^{-1}\bigg)&=\mathrm d\ast\bigg(\sum_i\sigma_i\otimes\sum_{j=1}^{10}\iota(L\phi_iL^{-1},E_j)E^*_j\bigg) \\ &=\mathrm d\ast\bigg(\sum_i\sigma_i\otimes\sum_{j=1}^{10}\iota(L\phi_iL^{-1},LL^{-1}E_jLL^{-1})E^*_j\bigg) \\ &=\mathrm d\ast\bigg(\sum_i\sigma_i\otimes    \sum_{j=1}^{10}\iota(\phi_i,L^{-1}E_jL)E^*_j\bigg) \ ,
\end{align*}where in the last step we used the invariance of $\iota$ under the action of $\mathcal C_\rho$ as explained in the proof of Lemma \ref{lem:invariantmetricg}. At this point, everything being explicit, it remains only to compute $L^{-1}E_jL$ for every $j=1,\dots,10$ and for $L$ equal to one matrix among those that can generate the centralizer $\mathcal C_\rho$ (see Lemma \ref{lem:centralizerorbifoldpoint}). We will explain only the case $L=A=\diag(-1,-1,-1,-1,1)=A^{-1}$, as the others are very similar, and will also need it in another proof later on. As just mentioned, with a straightforward computation we deduce that \begin{align*}
&A^{-1}E_1A=E_1,\quad A^{-1}E_2A=E_2,\quad A^{-1}E_3A=E_3,\quad A^{-1}E_4A=E_4, \\ & A^{-1}E_5A=-E_5, \quad  A^{-1}E_6A=E_6,\quad A^{-1}E_7A=-E_7,\quad A^{-1}E_8A=E_8,\\ & A^{-1}E_9A=-E_9,\quad A^{-1}E_{10}A=-E_{10} \ .
\end{align*}According to the above computation, we get $\mathrm d\ast\big(\sum_i\sigma_i\otimes\sum_{j=1}^{10}\iota(\phi_i,A^{-1}E_jA)E^*_j\big)=0$ since, up to a sign, the coefficients of $E^*_j$ coincides with those of (\ref{eq:coefficientidual}). Thus, the term $\mathrm d\ast\#\big(\sum_i\sigma_i\otimes A\phi_iA^{-1}\big)$ is equal to zero, which is equivalent to $\delta\big(\sum_i\sigma_i\otimes A\phi_iA^{-1}\big)=0$, as required.
\end{proof}
As for the connected component $\mathcal{M}_0(\Sg)$, the situation is slightly different in the sense that it contains points that are non-orbifold singularities. This means that the centralizer of the representation is no longer discrete but is a Lie subgroup of $\SO_0(2,3)$ with strictly positive dimension. The classification is as follows: 
\begin{proposition}[\cite{alessandrini2019geometry}]\label{prop:orbifoldpointsswequaltozero}
A maximal representation $\rho\in\mathcal{M}_0(\Sg)$ defines a \begin{itemize}
    \item[$\bullet$] a non-orbifold singularity if its Zariski closure is contained in a tightly embedded copy of $\SO_0(2,1)\times\SO(2)$;
    \item[$\bullet$] a $\Z_2$-orbifold singularity if its Zariski closure is contained in a tightly embedded copy of $\SO_0(2,2)$;
    \item[$\bullet$] smooth point otherwise.
\end{itemize}
\end{proposition}
Therefore, as can be deduced from the last result, the Fuchsian locus inside $\mathcal{M}_0(\Sg)$, and thus Teichm\"uller space, form a singularity with non-discrete centralizer. This does not allow us to apply the tools used previously for such points, but only for representations factoring through holonomies of GHMC anti-de Sitter $3$-manifolds. In fact, with exactly the same approach we get the following result: \begin{theorem}
Let $\rho\in\mathcal{M}_0(\Sg)$ and suppose that it is the holonomy of a GHMC anti-de Sitter $3$-manifolds isomorphic to $\Sg\times\R$, then the centralizer $$\mathcal C_\rho=\{\mathrm{Id}_5,A\}\cong\Z_2 , \quad A=\diag(-1,-1,-1,-1,1)$$acts by isometries on $H^1\big(\Sg,\lso_0(2,3)_{\Ad\rho}\big)$ with respect to $\g$.
\end{theorem}
\subsection{Totally geodesic sub-varieties and the Fuchsian locus}\label{sec:4.3}
Once we have been able to show that the restriction of the metric $\g$ on the subspaces of $\charvarietymaximal$ which represent orbifold points is well-defined, we want to show that in fact they embed as totally geodesic sub-varieties. In order to do so, we first find an isometry of the ambient space $\charvarietymaximal$, and then we show that the fixed points locus coincides with the subspace formed by those maximal representations factoring through one of the Lie group in the list of Section \ref{sec:4.2}. By a standard argument in Riemannian geometry, it follows that those subspaces are totally geodesic sub-varieties. Let us start with those representations whose Zariski closure is contained in $\SO_0(2,2)$. In this regard consider the map \begin{equation}\label{eq:isometryq}\begin{aligned}q:\mathfrak{R}^{\text{max}}_{2,3}&(\Sg)\longrightarrow\charvarietymaximal \\ & \rho\longmapsto Q\rho Q^{-1} \ ,\end{aligned}\end{equation}where $Q:=\diag(1,1,1,1,-1)\in\mathrm O(2,3)$. It is clear from the definition that the map $q$ fixes all the representations that are holonomies of GHMC anti-de Sitter $3$-manifolds isomorphic to $\Sg\times\R$, whose corresponding space will be denoted with $\GH(\Sg)$. Thus, we only need to prove that $q$ is an isometry for $\g$. The strategy will be similar to the one given in Section \ref{sec:4.2} with the appropriate differences and to the one given in the $\SL(3,\R)$ (\cite{li2013teichm}) and $\SO_0(2,2)$ case (\cite[\S 2.3]{tamburelli2021riemannian}). \\ \\ We are initially interested in understanding the induced map $q_\ast$ at the level of cohomology. In order to do this, we recall that (see \cite{goldman1984symplectic}) a tangent vector to a smooth path of representations $\rho_t$ is a map $u:\pi_1(\Sg)\to\lso_0(2,3)$ satisfying $$u(\gamma\gamma')-u(\gamma')=\Ad\big(\rho(\gamma)\big)u(\gamma') \ .$$ In particular, it is easy to see that if $u$ is tangent to $\rho$, then $QuQ^{-1}$ is tangent to $Q\rho Q^{-1}$. These tangent vectors are 1-cocycles representing a class in the group cohomology $H^1\big(\pi_1(\Sg),\lso_0(2,3)\big)$ which is isomorphic to $H^1\big(\Sg,\lso_0(2,3)_{\Ad\rho}\big)$, via the map \begin{equation}\begin{aligned}\label{eq:cohomologyisomorphic}
    H^1\big(&\Sg,\lso_0(2,3)_{\Ad\rho}\big)\to H^1\big(\pi_1(\Sg),\lso_0(2,3)\big) \\ &[\sigma\otimes\phi]\longmapsto\bigg(u_{\sigma\otimes\phi}:\gamma\mapsto \int_{\gamma}\sigma\otimes\phi\bigg) \ .
\end{aligned}\end{equation}\begin{lemma}\label{lem:qinducedoncohomology}For any $\sigma\in\Omega^1(\Sg)$ and for any section $\phi$ of $\lso_0(2,3)_{\Ad\rho}$, we have $$q_\ast[\sigma\otimes\phi]=[\sigma\otimes Q\phi Q^{-1}]$$
\end{lemma}
\begin{proof}
This is simply because for any $\gamma\in\pi_1(\Sg)$ $$\int_\gamma \sigma\otimes Q\phi Q^{-1}=Q\bigg(\int_\gamma\sigma\otimes\phi\bigg)Q^{-1} \ ,$$ hence $u_{\sigma\otimes Q\phi Q^{-1}}=Qu_{\sigma\otimes\phi}Q^{-1}$, which is exactly what we need according to the isomorphism (\ref{eq:cohomologyisomorphic})
\end{proof}
With abuse of notation, from here on we continue to denote by $q_\ast$ the induced map at the level of $\lso_0(2,3)_{\Ad\rho}$-valued $1$-forms.
\begin{lemma}\label{lem:invariantmetricgtotallygeodesic}
    For any $\sigma,\sigma'\in\Omega^1(\Sg)$ and for any section $\phi,\phi'$ of $\lso_0(2,3)_{\Ad\rho}$, we get $$g\big(q_\ast(\sigma\otimes\phi),q_\ast(\sigma'\otimes\phi')\big)=g\big(\sigma\otimes\phi,\sigma'\otimes\phi'\big) \ .$$
\end{lemma}
\begin{proof}
The argument of this proof differs slightly from the one given in Lemma \ref{lem:invariantmetricg}. In fact, in this case the matrix $Q$ belongs to $\mathrm O(2,3)$ and we can not appeal to the uniqueness statement in Theorem \ref{thm:rhoequivariantsurface}. Neverthless, given $\rho\in\charvarietymaximal$ if we denote by $\varphi(\Sg)_\rho$ the $\rho$-equivariant space-like maximal surface in $\Hypdue$ (Theorem \ref{thm:rhoequivariantsurface}), then it is isometric to the $q(\rho)$-equivariant maximal space-like surface $\varphi(\Sg)_{Q\rho Q^{-1}}$, with respect to the induced metrics $h$ and $h^q$ which coincide on every $\widetilde x\in\widetilde\Sg$. From this we deduce that if $H$ is a matrix representation of the $\rho$-equivariant inner product $\iota_{\widetilde x}$ on $\widetilde\Sg\times\lso_0(2,3)$, then $H^q:=Q^tHQ$ is the matrix representation of the $q(\rho)$-equivariant scalar product $\iota_{\widetilde x}^q$. After noting that $Q^t=Q=Q^{-1}$, for any $\phi,\phi'$ sections of $\lso_0(2,3)_{\Ad\rho}$ and for any $p\in\Sg$, we compute \begin{align*}
    \iota_p(\phi,\phi')&=\iota_{\widetilde x}(M,N) \quad \text{setting} \ M:=\widetilde\phi_{\widetilde x}, N:=\widetilde\phi'_{\widetilde x} \ \text{and} \ \widetilde x\in\pi^{-1}(p) \ \text{as in} \ (\ref{eq:iota}) \\ &=\tr(M^tH^{-1}NH) \quad \text{by Lemma} \ \ref{lem:matrixscalarproduct} \\ &=\tr\big(Q^t(Q^t)^{-1}M^tQ^t(Q^t)^{-1}QQ^{-1}H^{-1}(Q^t)^{-1}Q^tQ^{-1}QNQ^{-1}Q(Q^t)^{-1}Q^tHQQ^{-1}\big) \\ &=\tr\big(Q^t(q_\ast(M))^t(H^q)^{-1}q_\ast(N)H^qQ^{-1}\big) \\ &=\tr\big((q_\ast(M))^t(H^q)^{-1}q_\ast(N)H^q\big) \quad \text{by trace symmetry} \\ &=\iota^q_{\widetilde x}\big(q_\ast(M),q_\ast(N)\big) \\ &=\iota^q_p\big(q_\ast(\phi),q_\ast(\phi')\big) \quad \text{by Lemma} \ \ref{lem:qinducedoncohomology} \ .
\end{align*} In the end, we get \begin{align*}
g\big(\sigma\otimes\phi,\sigma'\otimes\phi'\big)&=\int_\Sg\iota(\phi,\phi')\sigma\wedge(\ast_h\sigma') \\ &=\int_\Sg\iota(\phi,\phi')(\sigma,\sigma')_h\mathrm d\mathrm{Vol}_h \\ &=\int_\Sg\iota^q(q_\ast(\phi),q_\ast(\phi'))(\sigma,\sigma')_{h^q}\mathrm d\mathrm{Vol}_{h^q} \\ &=g\big(q_\ast(\sigma\otimes\phi),q_\ast(\sigma'\otimes\phi')\big) \ .
\end{align*}We have therefore shown that $q_\ast$ is an isometry with respect the metric $g$ on the space of sections $\Omega^1(\Sg,\lso_0(2,3)_{\Ad\rho})$ and $\Omega^1(\Sg,\lso_0(2,3)_{\Ad q(\rho)})$.
\end{proof}
\begin{proposition}\label{prop:harmonicformsantidesitter}
    The induced map in cohomology $$q_\ast:H^1\big(\Sg,\lso_0(2,3)_{\Ad\rho}\big)\to H^1\big(\Sg,\lso_0(2,3)_{\Ad\rho}\big)$$ sends harmonic $1$-forms to harmonic $1$-forms.
\end{proposition}
\begin{proof}
In this case, unlike the previous one, we can make use of the technique used in Theorem \ref{thm:metriccompatiblewithorbifold}. The only difference is that now the maximal surfaces in $\Hypdue$ associated with $\rho$ and $q(\rho)$ are no longer equal but isometric. For this reason we will have to deal also with the analogous operator $\#^q$ defined on $\lso_0(2,3)_{\Ad q(\rho)}$-valued $1$-forms and with the inner product $\iota^q$ on $\lso_0(2,3)_{\Ad q(\rho)}$ introduced in the proof of Lemma \ref{lem:invariantmetricgtotallygeodesic}. Nevertheless, again we have to show that if $\sum_i\sigma_i\otimes\phi_i$ is a harmonic representative in its cohomology class and if $\mathrm d\big(\sum_i\sigma_i\otimes\phi_i\big)=\delta\big(\sum_i\sigma_i\otimes\phi_i\big)=0$, then the terms $\mathrm d\big(\sum_i\sigma_i\otimes Q\phi_iQ^{-1}\big)$ and $\delta\big(\sum_i\sigma_i\otimes Q\phi_iQ^{-1}\big)$ are also equal to zero. As usual, $\mathrm d$-closedness follows from the linearity of the differential, and as for $\delta$-closedness, however, we must show that $\mathrm d\ast\big(\sum_i\sigma_i\otimes\#^q(Q\phi_iQ^{-1})\big)=0$. Let $\{E_j\}_{j=1}^{10}$ be the basis of $\lso_0(2,3)$ introduced in the proof of Theorem \ref{thm:metriccompatiblewithorbifold}
and let $\{E_j^*\}_{j=1}^{10}$ be its dual basis. By definition, we have $$\#M=\sum_{j=1}^{10}\iota(M,E_j)E^*_j\quad \text{and}\quad \#^qM=\sum_{j=1}^{10}\iota^q(M,E_j)E^*_j \ .$$ In particular, the equation $\delta\big(\sum_i\sigma_i\otimes\phi_i\big)=0$ implies that $\mathrm d\ast\big(\sum_i\sigma_i\otimes\#\phi_i\big)=0$ which is equivalent to $$\mathrm d\ast\bigg(\sum_i\sigma_i\otimes\sum_{j=1}^{10}\iota(\phi_i,E_j)E_j^*\bigg)=0 \ .$$ The above relation allows us to conclude that $$\mathrm d\ast\bigg(\sum_i\sigma_i\iota(\phi_i,E_j)\bigg)=0 \ , \ \text{for any} \ j=1,\dots,10 \ .$$ At this point, using the equality $\iota^q(QMQ^{-1},QNQ^{-1})=\iota(M,N)$ for $M,N\in\lso_0(2,3)$ obtained in the proof of Lemma \ref{lem:invariantmetricgtotallygeodesic}, we get \begin{align*}
   \mathrm d\ast\bigg(\sum_i\sigma_i\otimes\#^q Q\phi_iQ^{-1}\bigg)&=\mathrm d\ast\bigg(\sum_i\sigma_i\otimes\sum_{j=1}^{10}\iota^q(Q\phi_iQ^{-1},E_j)E^*_j\bigg) \\ &=\mathrm d\ast\bigg(\sum_i\sigma_i\otimes\sum_{j=1}^{10}\iota^q(Q\phi_iQ^{-1},QQ^{-1}E_jQQ^{-1})E^*_j\bigg) \\ &=\mathrm d\ast\bigg(\sum_i\sigma_i\otimes    \sum_{j=1}^{10}\iota(\phi_i,Q^{-1}E_jQ)E^*_j\bigg) \ . 
\end{align*} Given that $Q=-A$, where $A:=\diag(-1,-1,-1,-1,1)$ we can refer to the computation performed in Theorem \ref{thm:metriccompatiblewithorbifold} and conclude that $\mathrm d\ast\#\big(\sum_i\sigma_i\otimes Q\phi_i Q^{-1}\big)=0$ and thus $\delta\big(\sigma_i\otimes Q\phi_iQ^{-1}\big)=0$, as required.
\end{proof}
Let us denote with $\g_\mathrm T$ the Riemannian metric defined on $\GH(\Sg)$ by Tamburelli (see Remark \ref{rem:metricSOLiTamburelli}). Combining everything together, we have the following:
\begin{theorem}\label{thm:antidesittertotallygeodesic}
The space $\big(\GH(\Sg),\g_\mathrm T\big)$ embeds as a totally geodesic sub-variety in $\charvarietymaximalstiefel$ and in $\mathcal{M}_0(\Sg)$ with respect to $\g$. In particular, the copy of Teichm\"uller space inside $\charvarietymaximalstiefel$ embeds as a totally geodesic sub-variety with respect to a multiple of the Weil-Petersson metric $\g_{\text{WP}}$.
\end{theorem}
\begin{proof}
The first claim follows from the fact that if $\rho\in\charvarietymaximal$ is the holonomy of a GHMC anti-de Sitter $3$-manifolds, then $$T_{[\rho]}\charvarietymaximal=\bigslant{H^1\big(\Sg,\lso_0(2,3)_{\Ad\rho}\big)}{\mathcal C_\rho} \ .$$ Since the metric $\g$ is invariant by the action of the centralizer $\mathcal C_\rho$, it can be restricted to the tangent to $\GH(\Sg)$, which is given by the inclusion of $H^1\big(\Sg,\lso_0(2,2)_{\Ad\rho}\big)$ inside $H^1\big(\Sg,\lso_0(2,3)_{\Ad\rho}\big)$ quotient out by $\mathcal C_\rho$. In particular, the restricted metric coincides with $\g_\mathrm T$ (\cite[\S 2.2]{tamburelli2021riemannian}). The map $q$ defined in (\ref{eq:isometryq}) is an isometry for $\charvarietymaximal$ with respect to $\g$ and the fixed locus is exactly $\GH(\Sg)$. Thus, we can conclude that $(\GH(\Sg),\g_\mathrm T)$ is totally geodesic in $\charvarietymaximalstiefel$ and in $\mathcal{M}_0(\Sg)$ with respect to $\g$. Regarding the second claim, we use the result proved by Tamburelli (\cite[Theorem 2.8]{tamburelli2021riemannian}) which states that $\g_\mathrm T$ restricts to a multiple of the Weil-Petersson metric on the Fuchsian locus, which embeds as a totally geodesic submanifold. Putting this result together with the first claim of the theorem, we obtain that the Fuchsian locus is totally geodesic in $\charvarietymaximalstiefel$ as well.
\end{proof}
\begin{remark}It must be pointed out that in light of Proposition \ref{prop:orbifoldpointsswequaltozero}, we can not conclude the Fuchsian locus is totally geodesic in $\mathcal{M}_0(\Sg)$, since it represents a non-orbifold point in the above connected component.\end{remark}
In the final part of this section, we want to present the argument that allows us to show that $\g$, effectively restricts to a multiple of the Weil-Petersson metric on the Fuchsian locus. Although we already know this to be true from Theorem \ref{thm:antidesittertotallygeodesic}, we want to explain the strategy of proof since it will relate back to the case of the Hitchin component in Section \ref{sec:4.5}. Once again, we follow the approach presented in \cite[\S 2.3]{tamburelli2021riemannian} and \cite{li2013teichm} with the appropriate differences. \\ \\ Suppose that $\rho\in\charvarietymaximalstiefel$ is in the embedded copy of Teichm\"uller space, namely $\rho=\xi\circ\rho_{Fuch}$ where $\rho_{\text{Fuch}}:\pi_1(\Sg)\to\SO_0(2,1)$ is discrete and faithful and $\xi$ is the standard inclusion given by: $$\xi:\SO_0(2,1)\longrightarrow\begin{pmatrix}
    \SO_0(2,1) & 0 \\ 0 & \mathrm{Id}_2
\end{pmatrix}\subset\SO_0(2,3) \ .$$ Such a representation $\rho$ preserves a totally geodesic space-like plane in $\Hypdue$. After explicitly realizing the double cover of the pseudo-hyperbolic space as $$\Hypduetilde=\{\underline x\in\R^{5} \ | \ x_1^2+x_2^2-x_3^2-x_4^2-x_5^2=-1\}$$ we can assume, up to post-composition by an isometry of the space, that $\rho$ preserves the hyperboloid $$\mathcal H=\{\underline x\in\R^5 \ | \ x_1^2+x_2^2-x_3^2=-1, \ x_4=x_5=0\} \ ,$$ which is isometric to the hyperbolic plane $\h^2=\{z\in\C \ | \ \Im(z)>0\}$ via the following map (\cite{li2013teichm}): \begin{equation}\begin{aligned}
    & f:\h^2\to\mathcal H\subset\R^5 \\ & (x,y)\mapsto\bigg(\frac{x}{y},\frac{x^2+y^2-1}{2y},\frac{x^2+y^2+1}{2y},0,0\bigg) \ .
\end{aligned}\end{equation}The standard copy of $\SO_0(2,1)$ inside $\SO_0(2,3)$ induced by $\xi$ is isomorphic to $\PSL(2,\R)$ via the map (\cite{kim2006embedding}): \begin{equation}\begin{aligned}
    \Phi: \ &\PSL(2,\R)\to\SO_0(2,1)<\SO_0(2,3) \\ & \begin{pmatrix}
        a & b \\ c & d
    \end{pmatrix}\mapsto\begin{pmatrix}
        ad+bc & ac-bd & ac+bd & 0 & 0 \\ ab-cd & \frac{a^2-b^2-c^2+d^2}{2} & \frac{a^2+b^2-c^2-d^2}{2} & 0 & 0 \\ ab+cd & \frac{a^2-b^2+c^2-d^2}{2} & \frac{a^2+b^2+c^2+d^2}{2} & 0 & 0 \\ 0 & 0 & 0 & 1 & 0 \\ 0 & 0 & 0 & 0 & 1
    \end{pmatrix} \ .
\end{aligned}\end{equation} Moreover, the induced map at the level of Lie algebra is given by: \begin{equation}\begin{aligned}
    \Phi_\ast & : \Lsl(2,\R)\to\lso_0(2,3) \\ & \begin{pmatrix}
        a & b \\ c & -a
    \end{pmatrix}\mapsto\begin{pmatrix}
        0 & c-b & c+b & 0 & 0 \\ b-c & 0 & 2a & 0 & 0 \\ b+c & 2a & 0 & 0 & 0 \\ 0 & 0 & 0 & 0 & 0 \\ 0 & 0 & 0 & 0 & 0
    \end{pmatrix} \ .
\end{aligned}\end{equation}
To each maximal representation $\rho:\pi_1(\Sg)\to\SO_0(2,3)$ is associated a unique $\rho$-equivariant maximal space-like embedding $\varphi:\widetilde\Sg\to\Hypdue$ (see Theorem \ref{thm:rhoequivariantsurface}). In our case, with $\rho=\xi\circ\rho_{\text{Fuch}}$ a Fuchsian representation, if we set $\Gamma:=\rho\big(\pi_1(\Sg)\big)<\SO_0(2,3)$, then the maximal space-like surface is realized as $\mathcal H/\Gamma$ and is isometric to the hyperbolic surface $\h^2/\Phi^{-1}(\Gamma)$. \\ \\ The metric $\g$ we defined on $H^1\big(\Sg,\lso_0(2,3)_{\Ad\rho}\big)$ depends on the choice of a hyperbolic metric $h$ on $\Sg$ and a scalar product $\iota$ on $\lso_0(2,3)_{\Ad\rho}$. Recall also that $\iota$ is determined by a family of scalar products $\{\iota_{\widetilde x}\}_{\widetilde x\in\widetilde\Sg}$ on $\R^5$, which are obtained by declaring the frame $\{u_1(\widetilde x),u_2(\widetilde x),\varphi(\widetilde x),N_1(\widetilde x),N_2(\widetilde x)\}$ orthonormal. Moreover, the map $f$ gives an explicit $\rho=\xi\circ\rho_{\text{Fuch}}$-equivariant maximal space-like embedding of $\widetilde\Sg$ into $\Hypdue$. After identifying the universal cover $\widetilde\Sg$ with $\h^2$, the coordinates of the tangent and normal vectors to the embedded surface can be computed with respect to the canonical basis of $\R^5$, so that the following matrix representation $H$ of $\iota_z$ can be obtained for any $z\in\h^2$ (\cite{li2013teichm}):
$$H=\begin{pmatrix}
    \frac{2x^2}{y^2}+1 & \frac{x(x^2+y^2-1)}{y^2} & -\frac{x(x^2+y^2+1)}{y^2} & 0 & 0 \\ \frac{x(x^2+y^2-1)}{y^2} & \frac{(x^2+y^2-1)^2}{2y^2}+1 & -\frac{(x^2+y^2+1)(x^2+y^2-1)}{2y^2} & 0 & 0 \\ -\frac{x(x^2+y^2+1)}{y^2} & -\frac{(x^2+y^2+1)(x^2+y^2-1)}{2y^2} & \frac{(x^2+y^2+1)^2}{2y^2}-1 & 0 & 0 \\ 0 & 0 & 0 & 1 & 0 \\ 0 & 0 & 0 & 0 & 1
\end{pmatrix}$$
whose inverse is given by $$H^{-1}=\begin{pmatrix}
    \frac{2x^2}{y^2}+1 & \frac{x(x^2+y^2-1)}{y^2} & \frac{x(x^2+y^2+1)}{y^2} & 0 & 0 \\ \frac{x(x^2+y^2-1)}{y^2} & \frac{(x^2+y^2-1)^2}{2y^2}+1 & \frac{(x^2+y^2+1)(x^2+y^2-1)}{2y^2} & 0 & 0 \\ \frac{x(x^2+y^2+1)}{y^2} & \frac{(x^2+y^2+1)(x^2+y^2-1)}{2y^2} & \frac{(x^2+y^2+1)^2}{2y^2}-1 & 0 & 0 \\ 0 & 0 & 0 & 1 & 0 \\ 0 & 0 & 0 & 0 & 1
\end{pmatrix} \ .$$
\begin{lemma}[\cite{li2013teichm}]\label{lem:matrixscalarproductcomplexversion} For any $z\in\h^2$, after extending the formula of Lemma \ref{lem:matrixscalarproduct} to $M,N\in\lso(5,\C)$ by $\iota_z(M,N)=\tr(M^tH^{-1}\widebar NH)$, we get
    $$\iota_z\bigg(\Phi_\ast\begin{pmatrix}
        -z & z^2 \\ -1 & z 
    \end{pmatrix},\Phi_\ast\begin{pmatrix}
        -z & z^2 \\ -1 & z 
    \end{pmatrix}\bigg)=16 y^2 \ .$$
\end{lemma}
\begin{proposition}\label{prop:tangenttoFuchsianlocus}
Let $\rho=\xi\circ\rho_{\text{Fuch}}\in\charvarietymaximalstiefel$ be a Fuchsian representation, then \begin{itemize}
    \item[(1)] the tangent space at $\rho$ to the Fuchsian locus is spanned by the cohomology class of $\psi(z)\mathrm dz\otimes\Phi_\ast\begin{psmallmatrix}
        -z & z^2 \\ -1 & z 
    \end{psmallmatrix}$, where $\psi(z)\mathrm dz^2$ is a holomorphic quadratic differential on $\mathcal H/\Gamma\cong\h^2/\Phi^{-1}(\Gamma)$;
    \item[(2)] the $\lso_0(2,3)_{\Ad\rho}$-valued $1$-forms $\psi(z)\mathrm dz\otimes\Phi_\ast\begin{psmallmatrix}
        -z & z^2 \\ -1 & z 
    \end{psmallmatrix}$ are harmonic representatives in their own cohomology class. 
\end{itemize}
\end{proposition}
\begin{proof}
First recall that if $\rho$ is in the copy of Teichm\"uller space, then it is a $\Z_2\times\Z_2$-orbifold point in $\charvarietymaximalstiefel$ (Proposition \ref{prop:orbifoldHiggs}). In particular, $$T_{[\rho]}\charvarietymaximalstiefel\cong\bigslant{H^1(\Sg,\lso_0(2,3)_{\Ad\rho})}{\mathcal C_\rho} \ , $$ where the action of the centralizer on the first cohomology group is given by conjugation on the matrix part. Notice that, if $\psi(z)\mathrm dz^2$ is a holomorphic quadratic differential on $\mathcal H/\Gamma\cong\h^2/\Phi^{-1}(\Gamma)$ then the $\lso_0(2,3)_{\Ad\rho}$-valued $1$-form $\psi(z)\mathrm dz\otimes\Phi_\ast\begin{psmallmatrix}
        -z & z^2 \\ -1 & z 
    \end{psmallmatrix}$ is invariant by the action of $\Phi\big(\mathcal C_\rho\big)=\{A,B,C,\mathrm{Id}_5\}$, where $A=\diag(-1,-1,-1,-1,1), B=\diag(1,1,1,-1,-1)$ and $C=A\cdot B=\diag(-1,-1,-1,1,-1)$ (see Lemma \ref{lem:centralizerorbifoldpoint}). \newline
$(1)$ Now let us consider the corresponding Fuchsian representation $\widetilde\rho:=\Phi^{-1}(\rho)$ into $\PSL(2,\R)$. The claim is obtained from the $\mathcal C_\rho$-invariance describe above and from the fact (\cite{goldman1984symplectic}) that the tangent space to Teichm\"uller space is generated by the $\Lsl(2,\R)_{\Ad\widetilde\rho}$-valued $1$-forms $\psi(z)\mathrm dz\otimes\begin{psmallmatrix}
        -z & z^2 \\ -1 & z 
    \end{psmallmatrix}$ and thus, the tangent space to the Fuchsian locus is generated by the inclusion of $H^1\big(\Sg,\Lsl(2,\R)_{\Ad\widetilde\rho}\big)$ inside $H^1\big(\Sg,\lso_0(2,3)_{\Ad\rho}\big)$ induced by $\Phi_\ast$.\newline $(2)$ Requiring $\psi(z)\mathrm dz\otimes\Phi_\ast\begin{psmallmatrix}
        -z & z^2 \\ -1 & z 
    \end{psmallmatrix}$ to be harmonic is equivalent to prove that it is $\mathrm d$-closed and $\delta$-closed. The argument for the first claim can be found in \cite[Lemma 5]{li2013teichm} and it applies the same way in our case. Regarding $\delta$-closedness, we follow the strategy of the above lemma. As $\delta$ was defined (Section \ref{sec:3.1}), it is sufficient to show that $\mathrm d\ast\#\bigg(\psi(z)\mathrm dz\otimes\Phi_\ast\begin{psmallmatrix}
        -z & z^2 \\ -1 & z 
    \end{psmallmatrix}\bigg)=0$. By linearity, \begin{align*}
        \#\bigg(\psi(z)\mathrm dz\otimes\Phi_\ast\begin{psmallmatrix}
        -z & z^2 \\ -1 & z 
    \end{psmallmatrix}\bigg)&=z^2\psi(z)\mathrm dz\otimes\#\big(\Phi_\ast\begin{psmallmatrix}
        0 & 1 \\ 0 & 0 
    \end{psmallmatrix}\big)-\psi(z)\mathrm dz\otimes\#\big(\Phi_\ast\begin{psmallmatrix}
        0 & 0 \\ 1 & 0
    \end{psmallmatrix}\big) \\ &-2z\psi(z)\mathrm dz\otimes\#\bigg(\Phi_\ast\begin{psmallmatrix}
        \frac{1}{2} & 0 \\ 0 & -\frac{1}{2}   \end{psmallmatrix}\bigg) \ .
    \end{align*} Let $\{E_j\}_{j=1}^{10}$ be the basis for $\lso_0(2,3)$ introduced in the proof of Theorem \ref{thm:metriccompatiblewithorbifold}, and notice that $$E_1=\Phi_\ast\begin{psmallmatrix}
        0 & 1 \\ 0 & 0 
    \end{psmallmatrix}=\begin{psmallmatrix}
        0 & -1 & 1 & 0 & 0 \\ 1 & 0 & 0 & 0 & 0 \\ 1 & 0 & 0 & 0 & 0 \\ 0 & 0 & 0 & 0 & 0 \\ 0 & 0 & 0 & 0 & 0
    \end{psmallmatrix},\quad E_2=\Phi_\ast\begin{psmallmatrix}
        \frac{1}{2} & 0 \\ 0 & -\frac{1}{2}   \end{psmallmatrix}=\begin{psmallmatrix}
        0 & 0 & 0 & 0 & 0 \\ 0 & 0 & 1 & 0 & 0 \\ 0 & 1 & 0 & 0 & 0 \\ 0 & 0 & 0 & 0 & 0 \\ 0 & 0 & 0 & 0 & 0
    \end{psmallmatrix},$$ $$E_3=\Phi_\ast\begin{psmallmatrix}
        0 & 0 \\ 1 & 0
    \end{psmallmatrix}=\begin{psmallmatrix}
        0 & 1 & 1 & 0 & 0 \\ -1 & 0 & 0 & 0 & 0 \\ 1 & 0 & 0 & 0 & 0 \\ 0 & 0 & 0 & 0 & 0 \\ 0 & 0 & 0 & 0 & 0
    \end{psmallmatrix} \ .$$ The operator $\#$ is given by $\#M=\sum_{i=1}^{10}\iota(M,E_i)E^*_i$, where $E^*_i$ is defined by setting $E^*_i(E_j)=\delta_i^j$. Using Lemma \ref{lem:matrixscalarproduct} to compute $\iota(E_i,E_j)$, we get \begin{align*}
        &\#E_1=\frac{4}{y^2}\big(E_1^*-x^2E_3^*+xE_2^*\big) \\ & \#E_2=\frac{4}{y^2}\big(xE_1^*-x(x^2+y^2)E_3^*-x(x^2+y^2)E_2^*\big) \\ & \#E_3=\frac{4}{y^2}\big(-x^2E_1^*+(x^2+y^2)E_3^*-x(x^2+y^2)E_2^*\big) \ .
    \end{align*} We note that $z$ is a conformal coordinate for the induced metric on the $\rho$-equivariant maximal surface in $\Hypdue$, thus from the definition of Hodge star operator we get $\ast\mathrm dx=\mathrm dy$ and $\ast\mathrm dy=-\mathrm dx$. In particular, after extending the operator to complex $1$-forms by complex anti-linearity ($\ast(i\alpha)=-i\ast\bar\alpha$), we obtain that $\ast\psi(z)\mathrm dz=i\overline{\psi(z)}\mathrm d\bar z$. Since $\psi(z)$
 is holomorphic and $\mathrm d=\partial+\bar\partial$, we have \begin{align*}
     &\mathrm d\ast\Big(-4\psi(z)\mathrm dz\otimes E_1^*+4z^2\psi(z)\mathrm dz\otimes E_3^*-4z\psi(z)\mathrm dz\otimes E_2^*\Big) \\ &=\mathrm d\Big(-4i\overline{\psi(z)}\mathrm d\bar z\otimes E_1^*+4i\overline{\psi(z)}\bar z^2\mathrm d\bar z\otimes E_3^*-4i\overline{\psi(z)}\bar z\mathrm d\bar z\otimes E_2^*\Big)=0 \ .
 \end{align*}In other words, $\psi(z)\mathrm dz\otimes\Phi_\ast\begin{psmallmatrix}
        -z & z^2 \\ -1 & z 
    \end{psmallmatrix}$ is $\delta$-closed, hence harmonic.
 \end{proof}
\begin{theorem}\label{thm:metriconFuchsianlocus}
    The metric $\g$ on the Fuchsian locus inside $\charvarietymaximalstiefel$ restricts to a multiple of the Weil-Petersson metric on Teichm\"uller space.
\end{theorem}
\begin{proof}
By Lemma \ref{prop:tangenttoFuchsianlocus} and by (\ref{eq:definitiongoncohomology}), it is enough to prove that $$g\bigg(\psi(z)\mathrm dz\otimes\Phi_\ast\begin{psmallmatrix}
        -z & z^2 \\ -1 & z 
    \end{psmallmatrix},\psi'(z)\mathrm dz\otimes\Phi_\ast\begin{psmallmatrix}
        -z & z^2 \\ -1 & z 
    \end{psmallmatrix}\bigg)=k\cdot\g_{\text{WP}}(\psi,\psi'), \quad k\in\R$$ where, by abuse of notation, we are denoting with $g$ its extension to a hermitian metric on $\lso_0(5,\C)$-valued $1$-forms. According to the definition of $g$ (see (\ref{eq:definitiong})) and Lemma \ref{lem:matrixscalarproductcomplexversion}, we get \begin{align*}
        &g\bigg(\psi(z)\mathrm dz\otimes\Phi_\ast\begin{psmallmatrix}
        -z & z^2 \\ -1 & z 
    \end{psmallmatrix},\psi'(z)\mathrm dz\otimes\Phi_\ast\begin{psmallmatrix}
        -z & z^2 \\ -1 & z 
    \end{psmallmatrix}\bigg) \\ &=\Ree\int_\Sg\iota_z\bigg(\Phi_\ast\begin{psmallmatrix}
        -z & z^2 \\ -1 & z 
    \end{psmallmatrix},\Phi_\ast\begin{psmallmatrix}
        -z & z^2 \\ -1 & z 
    \end{psmallmatrix}\bigg)\psi(z)\mathrm dz\wedge\ast\big(\psi'(z)\mathrm dz\big) \\ &=\Ree\int_\Sg\iota_z\bigg(\Phi_\ast\begin{psmallmatrix}
        -z & z^2 \\ -1 & z 
    \end{psmallmatrix},\Phi_\ast\begin{psmallmatrix}
        -z & z^2 \\ -1 & z 
    \end{psmallmatrix}\bigg)\psi(z)\mathrm dz\wedge\big(i\overline{\psi'(z)}\mathrm d\bar z\big) \\ &=\Ree\int_\Sg 16i\psi(z)\overline{\psi'(z)}y^2\mathrm dz\wedge\mathrm d\bar z \\ &=32\g_{\text{WP}}(\psi,\psi') \ .
    \end{align*}
\end{proof}

\subsection{A note about the Hitchin component}\label{sec:4.5}
Let $\Hit(\Sg)$ be the Hitchin component for $\SO_0(2,3)$ (see Section \ref{sec:4.1}), and recall that it is defined as the connected component of $\charvarietymaximal$ consisting of all surface group representations into $\SO_0(2,3)$ that can be deformed to Fuchsian ones, namely those that can be written as $\Phi\circ\rho_{\text{Fuch}}$ where $\rho_{Fuch}$ is a Fuchsian representation into $\PSL(2,\R)$ and $\Phi$ is the unique irreducible representation of $\PSL(2,\R)$ in $\SO_0(2,3)$. The Fuchsian representations in $\Hit(\Sg)$ form a submanifold $\mathcal F(\Sg)$, called the Fuchsian locus, which is isomorphic to a copy of Teichm\"uller space of the surface. Moreover, having a smooth manifold structure, $\Hit(\Sg)$ carries a well-defined metric $\g$ (Theorem \ref{thm:metricGothencomponentHitchin}). \begin{theorem}
The Fuchsian locus $\mathcal F(\Sg)$ endowed with the restricted metric $\g|_{\mathcal F(\Sg)}$ is a totally geodesic submanifold of $\big(\Hit(\Sg), \g\big)$, and $\g|_{\mathcal F(\Sg)}$ is not (a multiple of) the Weil-Petersson metric on Teichm\"uller space.
\end{theorem} It is somewhat surprising that the restriction of $\g$ to the Fuchsian locus is not a multiple of the Weil-Petersson metric, unlike in the case of $\SL(3,\R)$ (\cite{li2013teichm}). In what follows, we will not give all the details of the proof of the above theorem but the strategy that led to the formulation of its statement, which can be compared to what we explained (Section \ref{sec:4.3}) for the Fuchsian locus in the connected component $\charvarietymaximalstiefel$. \\ \\ \underline{\emph{Step 1: the $\rho$-equivariant parametrization of $\widetilde\Sg$ as a maximal surface in $\Hypdue$, for $\rho\in\mathcal F(\Sg)$}}\vspace{0.5em}\newline $\bullet$ Let us denote with $\R_2[x,y]$ the real vector space of degree two homogeneous polynomials in two variables, with a basis given by $\{x^2,xy,y^2\}$. On this space, one can introduce a non-degenerate bi-linear form of signature $(2,1)$ by $b_2(v,w)=\frac{1}{2}v_2w_2-v_1w_3-v_3w_1$, in such a way that the image of the irreducible representation \begin{align*}
    \widetilde\tau:& \ \PSL(2,\R)\longrightarrow\PSL(3,\R) \\ &\begin{pmatrix}
        a & b \\ c & d
    \end{pmatrix}\longmapsto \begin{pmatrix}
        a^2 & ab & b^2 \\ 2ac & ad+bc & 2bd \\ c^2 & cd & d^2
    \end{pmatrix}
\end{align*}preserves the associated quadratic form $Q_2:=\begin{psmallmatrix}
    0 & 0 & -1 \\ 0 & 1/2 & 0 \\ -1 & 0 & 0
\end{psmallmatrix}$ in the above basis. In other words, for any $A\in\PSL(2,\R)$ the matrix $\widetilde\tau(A)$ belongs to $\widetilde{\SO}_0(2,1)$ which is the identity component of $$\widetilde{\SO}(2,1):=\{M\in\SL(3,\R) \ | \ M^tQ_2M=Q_2\} \ .$$ Such a representation preserves a space-like plane in $\R^3$ which, up to post-composing by an isometry, is the twisted hyperboloid $\widetilde{\mathcal H}:=\{\underline x\in\R^3 \ | \ \frac{1}{2}x_2^2-2x_2x_3=-1\}$. In order to find the associated parameterization of $\h^2$, up to post-composing by an isometry, we can choose where to send point $i\in\h^2$ and then impose $\widetilde\tau$-equivariance. Indeed, by taking the matrix $A:=\frac{1}{\sqrt y}\begin{psmallmatrix}
    y & x \\ 0 & 1
\end{psmallmatrix}$ which sends $i\in\h^2$ to $z=x+iy\in\h^2$, we first define $\widetilde g:\h^2\to\widetilde{\mathcal H}\subset\R^3$ by $\widetilde g(i):=(\frac{\sqrt 2}{2},0,\frac{\sqrt 2}{2})$ and then, by $\widetilde\tau$-equivariance $$\widetilde g(x,y)=\widetilde g\big(A\cdot i\big)=\widetilde\tau(A)\cdot \widetilde g(i)=\frac{\sqrt 2}{2}\bigg(\frac{x^2+y^2}{2},2\frac{x}{y},\frac{1}{y}\bigg) \ .$$ $\bullet$ Now let $\R_4[x,y]$ be the real vector space of degree four homogeneous polynomials in two variables, with a basis given by $\{x^4,x^3y,x^2y^2,xy^3,y^4\}$. We can introduce a non-degenerate bi-linear form of signature $(2,3)$ by $b_4(v,w):=-\frac{1}{6}v_3w_3-v_1w_5-v_5w_1+\frac{1}{4}v_2w_4+\frac{1}{4}v_4w_2$ and consider the Lie group $\widetilde{\SO}_0(2,3)$ to be the identity component of $$\widetilde{\SO}(2,3):=\{M\in\SL(5,\R) \ | \ M^tQ_4M=Q_4\} \ \ \text{where} \  \ Q_4=\begin{psmallmatrix}
    0 & 0 & 0 & 0 & -1 \\ 0 & 0 & 0 & 1/4 & 0 \\ 0 & 0 & -1/6 & 0 & 0 \\ 0 & 1/4 & 0 & 0 & 0 \\ -1 & 0 & 0 & 0 & 0
\end{psmallmatrix}$$is the matrix associated with $b_4$ in the above basis. After computing the irreducible representation $\widetilde j:\widetilde{\SO}_0(2,1)\to\widetilde{\SO}_0(2,3)$ we can identify $\widetilde{\mathcal{H}}$ (resp. $\Hypdue_{b_4}$) with the projectivization of $\R_2[x,y]_{<0}$ (resp. $\R_4[x,y]_{<0}$), where the subscript $<0$ stands for polynomials of negative discriminant. Thus, we can define the map $\widetilde f:\widetilde{\mathcal{H}}\to\Hypdue_{b_4}$ which sends the equivalence class $[P]$ to $[P^2]$ (see also \cite[\S 5.3]{collier2019geometry}) and turns out to be $\widetilde j$-equivariant. In particular, $\widetilde f(\widetilde{\mathcal{H}})\subset\Hypdue_{b_4}$ can be directly computed and represents a space-like maximal surface referred to as the \emph{Veronese surface} (see \cite{ishihara1988maximal}). Everything being explicit, one can finally compute the composition $\widetilde F:=\widetilde f\circ\widetilde g:\h^2\to\Hypdue_{b_4}$ which turns out to be $(\widetilde j\circ\widetilde\tau)$-equivariant.\vspace{0.5em}\newline $\bullet$ Now, we want to bring back the parameterization $\widetilde F$ of the equivariant maximal surface in the standard pseudo-hyperbolic space model. After straightforward but very long computations, the irreducible representation $\Phi:\PSL(2,\R)\to\SO_0(2,3)$ is such that $$\Phi:\frac{1}{\sqrt y}\begin{pmatrix}
    y & x \\ 0 & 1
\end{pmatrix}\longmapsto \begin{pmatrix}
    \frac{1+3x^2+y^2}{2y} & -\frac{x(1+x^2)}{y^2} & \sqrt 3x & -\frac{-1-3x^2+y^2}{2y} & \frac{x(1+x^2)}{y^2} \\ xy+\frac{x^3}{y} & \frac{1-x^4+y^4}{2y^2} & \sqrt 3x^2 & xy-\frac{x^3}{y} & \frac{-1+x^4+y^4}{2y^2} \\ \frac{\sqrt 3x}{y} & -\frac{\sqrt 3x^2}{y^2} & 1 & -\frac{\sqrt 3x}{y} & \frac{\sqrt 3x^2}{y^2} \\ \frac{-1+3x^2+y^2}{2y} & \frac{x(1-x^2)}{y^2} & \sqrt 3x & \frac{1-3x^2+y^2}{2y} & \frac{x(-1+x^2)}{y^2} \\ xy+\frac{x^3}{y} & -\frac{1+x^4-y^4}{2y^2} & \sqrt 3x^2 & xy-\frac{x^3}{y} & \frac{1+x^4+y^4}{2y^2}
\end{pmatrix} \ .$$ Again, since in the twisted model we chose where to send the point $i\in\h^2$ and we imposed equivariance, the map $F:\h^2\to\Hypdue$ in the standard model is given by \begin{align*}F(x,y)=\Phi(A)\cdot F(i)=\frac{\sqrt 3}{2y^2}\bigg(&x(1+x^2+y^2), \frac{-1+(x^2+y^2)^2}{2}, \frac{\sqrt 3}{3}(y^2+3x^2), \\ & x(-1+x^2+y^2), \frac{1+(x^2+y^2)^2}{2}\bigg) \ .\end{align*} In other words, if $\rho\in\mathcal F(\Sg)$ is in the Fuchsian locus, after identifying $\widetilde\Sg$ with $\h^2$, we found the explicit parameterization of the unique $\rho$-equivariant maximal space-like embedding $F:\h^2\to\Hypdue$. \vspace{0.5em}\newline\underline{\emph{Step 2: the orthonormal frame in $\R^5$ and the matrix of the scalar product $\iota$}}\vspace{0.5em}\newline
Recall that the metric $\g$ was constructed from a metric $h$ on the surface $\Sg$ (which in this case is the induced metric as a maximal surface in $\Hypdue$) and a scalar product $\iota$ on $\lso_0(2,3)_{\Ad\rho}$. Such an inner product $\iota$ is determined by a family of scalar products $\iota_{\widetilde x}$ in $\R^5$, depending on $x\in\widetilde\Sg$, which are obtained by declaring the frame $\{u_1(\widetilde x),u_2(\widetilde x),F(\widetilde x),N_1(\widetilde x),N_2(\widetilde x)\}$ to be orthonormal, where $u_1$ and $u_2$ are the tangent vectors to the surface and $N_1,N_2$ are the normals. After identifying $\widetilde\Sg$ with $\h^2$, the $\rho$-equivariant map $F$ allows us to compute the vectors $u_1,u_2$ simply by deriving the position vector $F(x,y)$ in $x$ and $y$, respectively, and then normalize them to norm $1$. Having done this, we first compute the normal vectors at the point $i\in\h^2$, and then use the $\rho$-equivariance to obtain $N_1,N_2$ at any point $(x,y)$. In the end, the aforementioned vectors are given by $$ u_1=\begin{psmallmatrix}
    \frac{1+3x^2+y^2}{2y} \\ \frac{x}{y}(y^2+x^2) \\ \frac{\sqrt 3 x}{y} \\ \frac{-1+3x^2+y^2}{2y} \\ \frac{x}{y}(y^2+x^2)
\end{psmallmatrix}, \ u_2=\begin{psmallmatrix}
    -\frac{x(1+x^2)}{y^2} \\ \frac{1-x^4+y^4}{2y^2} \\ -\frac{\sqrt 3x^2}{y^2} \\ \frac{x(1-x^2)}{y^2} \\ -\frac{1-x^4-y^4}{2y^2}
\end{psmallmatrix}, \ N_1=\begin{psmallmatrix}
    \frac{-1-3x^2+y^2}{2y} \\ \frac{x}{y}(y^2-x^2) \\ -\frac{\sqrt 3x}{y} \\ \frac{1-3x^2+y^2}{2y} \\ \frac{x}{y}(y^2-x^2)
\end{psmallmatrix}, \ N_2=\begin{psmallmatrix}
    \frac{x}{2y^2}(x+x^2-3y^2) \\ \frac{-1+x^4-6x^2y^2+y^4}{4y^2} \\ \frac{\sqrt 3}{2}\Big(\frac{x^2}{y^2}-1\Big) \\ \frac{x}{2y^2}(-1+x^2-3y^2) \\ \frac{1+x^4-6x^2y^2+y^4}{4y^2}
\end{psmallmatrix} \ .$$ In particular, we can explictly compute the matrix $H$ representing $\iota_z$ in the canonical basis of $\R^5$, for $z\in\h^2$, and its inverse.\footnote{We decided not to insert the full matrices since their expression is too complicated.}\vspace{0.5em}\newline\underline{\emph{Step 3: the tangent space to the Fuchsian locus in $\Hit(\Sg)$}}\vspace{0.5em}\newline The next step is to prove something similar to Proposition \ref{prop:tangenttoFuchsianlocus}. First, if $\Phi$ is the irreducible representation of $\PSL(2,\R)$ into $\SO_0(2,3)$, the map induced at the level of Lie algebras is \begin{align*}
    & \ \ \ \Phi_\ast:\Lsl(2,\R)\longrightarrow\lso_0(2,3) \\ &\begin{pmatrix}
        a & b \\ c & -a
    \end{pmatrix}\longmapsto\begin{pmatrix}
        0 & c-b & \sqrt 3(b+c) & 2a & b+c \\ b-c & 0 & 0 & b+c & 4a \\ \sqrt 3(b+c) & 0 & 0 &\sqrt 3(c-b) & 0 \\ 2a & b+c & \sqrt 3(b-c) & 0 & c-b \\ b+c & 4a & 0 & b-c & 0
    \end{pmatrix} \ .
\end{align*} The first claim of the aforementioned proposition is developed in the same way and thus needs no further explanation. As for the second point, however, the situation is quite different. In fact, for every holomorphic quadratic differential $\psi(z)\mathrm dz^2$ on $\h^2/\Phi^{-1}(\Gamma)$, it must be shown that the $\lso_0(2,3)_{\Ad\rho}$-valued $1$-form $\psi(z)\mathrm dz\otimes\Phi_\ast\begin{psmallmatrix}
    -z & z^2 \\ -1 & z 
\end{psmallmatrix}$ is a harmonic representative in its cohomology class, that is, it is both $\mathrm d$-closed and $\delta$-closed. It all boils down to computing $$z^2\psi(z)\mathrm dz\otimes\#E_1-\psi(z)\mathrm dz\otimes\#E_3-2z\psi(z)\mathrm dz\otimes\#E_2 \ ,$$ where \begin{align*}&E_1:=\Phi_\ast\begin{psmallmatrix}
    0 & 1 \\ 0 & 0 
\end{psmallmatrix}=\begin{psmallmatrix}
    0 & -1 & \sqrt 3 & 0 & 1 \\ 1 & 0 & 0 & 1 & 0 \\ \sqrt 3 & 0 & 0 & -\sqrt 3 & 0 \\ 0 & 1 & \sqrt 3 & 0 & -1 \\ 1 & 0 & 0 & 1 & 0 
\end{psmallmatrix},\quad E_2:=\Phi_\ast\begin{psmallmatrix}
    \frac{1}{2} & 0 \\ 0 & -\frac{1}{2}
\end{psmallmatrix}=\begin{psmallmatrix}
    0 & 0 & 0 & 1 & 0 \\ 0 & 0 & 0 & 0 & 2 \\ 0 & 0 & 0 & 0 & 0 \\ 1 & 0 & 0 & 0 & 0 \\ 0 & 2 & 0 & 0 & 0 
\end{psmallmatrix},\\ & E_3=\Phi_\ast\begin{psmallmatrix}
    0 & 0 \\ 1 & 0
\end{psmallmatrix}=\begin{psmallmatrix}
    0 & 1 & \sqrt 3 & 0 & 1 \\ -1 & 0 & 0 & 1 & 0 \\ \sqrt 3 & 0 & 0 & \sqrt 3 & 0 \\ 0 & 1 & -\sqrt 3 & 0 & -1 \\ 1 & 0 & 0 & -1 & 0 
\end{psmallmatrix} \ ,\end{align*} and then concluding $\mathrm d\ast\big(z^2\psi(z)\mathrm dz\otimes\#E_1-\psi(z)\mathrm dz\otimes\#E_3-2z\psi(z)\mathrm dz\otimes\#E_2\big)=0$. Since the induced map $\Phi_\ast$ is different, we have to complete $\{E_1,E_2,E_3\}$ to a basis for $\lso_0(2,3)$ like the one in the proof of Theorem \ref{thm:metriccompatiblewithorbifold} except by making a substitution of these two matrices \begin{align*}
    E_4=\begin{psmallmatrix}
        0 & 1 & 0 & 0 & 0 \\ -1 & 0 & 0 & 0 & 0 \\ 0 & 0 & 0 & 0 & 0 \\ 0 & 0 & 0 & 0 & 0 \\ 0 & 0 & 0 & 0 & 0 
    \end{psmallmatrix},\quad E_7=\begin{psmallmatrix}
        0 & 0 & 0 & 0 & 0 \\ 0 & 0 & 1 & 0 & 0 \\ 0 & 1 & 0 & 0 & 0 \\ 0 & 0 & 0 & 0 & 0 \\ 0 & 0 & 0 & 0 & 0 
    \end{psmallmatrix} \ .
\end{align*} So, it remains only to compute the three terms $\#E_1,\#E_2$ and $\#E_3$ using that $$\#M=\sum_{i=1}^{10}\iota(M,E_i)E^*_i,\quad \iota (M,E_i)=\tr(M^tH^{-1}E_iH)$$ and knowing the explicit expression of the matrices $H,H^{-1}$ computed in step 2. We get that \begin{align*}
    \mathrm d\ast&\Big(z^2\psi(z)\mathrm dz\otimes\#E_1-\psi(z)\mathrm dz\otimes\#E_3-2z\psi(z)\mathrm dz\otimes\#E_2\Big)=\mathrm d\ast\bigg(\psi(z)\Big(20z^2\mathrm dz\otimes E^*_1 \\ &-20z\mathrm dz\otimes E_2^*-20\mathrm dz\otimes E_3^*-2(1+z^2)\mathrm dz\otimes E_4^* -2(1+z)\mathrm dz\otimes E_5^*+2(z^2-1)\mathrm dz\otimes E_8^* \\ & -4z\mathrm dz\otimes E_9^*+2(z^2-1)\mathrm dz\otimes E_{10}^*\Big)\bigg)=0 \ .
\end{align*}
\underline{\emph{Step 4: the Fuchsian locus is totally geodesic}}\vspace{0.5em}\newline The idea is still the same, which is to find an isometry of $\Hit(\Sg)$ that has exactly the Fuchsian locus as fixed points. Unlike the anti-de Sitter case or the Fuchsian locus in $\charvarietymaximalstiefel$, it is not clear how to find such a map by looking directly at representations since $\SO_0(2,1)$ is not trivially embedded in $\SO_0(2,3)$. In this case, however, thanks to the outstanding theorem proved by Labourie (\cite{labourie2017cyclic}), we know that $\Hit(\Sg)$ is mapping class group equivariantly isomorphic to the holomorphic bundle of quartic differential over Teichm\"uller space. Thanks to this identification, each point of $\Hit(\Sg)$ can be thought of as a pair $(J,q_4)$ where $J$ is a complex structure on $\Sg$ and $q_4$ is a $J$-holomorphic quartic differential, that is, a holomorphic section of $K_X^{\otimes^4}$, where $X=(\Sg,J)$. In this way, the fixed points locus of \begin{align*}
    &\Hit(\Sg)\to\Hit(\Sg) \\ & (J,q_4)\mapsto (J,-q_4)
\end{align*} are exactly the pairs $(J,q_4)$ with $q_4=0$, i.e. the Fuchsian locus. To this end we need to understand how the induced map acts at the level of representations in order to use the same approach as in Proposition \ref{prop:harmonicformsantidesitter}. It can be shown\footnote{Studying the relation between $q_4$ and the orthonormal frame in $\R^5$ defined by the maximal surface, it can be seen that changing the sign to $q_4$ is equivalent to changing the sign to one of the two normal vectors.} that switching the sign to the quartic differential is equivalent to conjugating the associated representation for $Q=\diag(1,1,1,1,-1)\in\mathrm O(2,3)$. Now, one may argue as in Section \ref{sec:4.3} and show that $\big(\mathcal F(\Sg),\g|_{\mathcal F(\Sg)}\big)$ is totally geodesic in $\big(\Hit(\Sg),\g\big)$.\vspace{0.5em}\newline\underline{\emph{Step 5: the restricted metric on $\mathcal F(\Sg)$ is not a multiple of the Weil-Petersson one}}\vspace{0.5em}\newline
It remains to explain the analogous version of Lemma \ref{lem:matrixscalarproductcomplexversion} which then leads to the computation of the restricted metric on the Fuchsian locus as in Theorem \ref{thm:metriconFuchsianlocus}. It all depends on the following result $$\iota_z\bigg(\Phi_\ast\begin{pmatrix}
        -z & z^2 \\ -1 & z 
    \end{pmatrix},\Phi_\ast\begin{pmatrix}
        -z & z^2 \\ -1 & z 
    \end{pmatrix}\bigg)=\frac{20\big(1+x^4+y^4+2x^2(1+y^2)\big)^2}{y^2} \ ,$$which by the same strategy as in the aforementioned theorem, leads us to conclude that \begin{align*}
        &g\bigg(\psi(z)\mathrm dz\otimes\Phi_\ast\begin{psmallmatrix}
        -z & z^2 \\ -1 & z 
    \end{psmallmatrix},\psi'(z)\mathrm dz\otimes\Phi_\ast\begin{psmallmatrix}
        -z & z^2 \\ -1 & z 
    \end{psmallmatrix}\bigg) \\ &=\Ree\int_\Sg\iota_z\bigg(\Phi_\ast\begin{psmallmatrix}
        -z & z^2 \\ -1 & z 
    \end{psmallmatrix},\Phi_\ast\begin{psmallmatrix}
        -z & z^2 \\ -1 & z 
    \end{psmallmatrix}\bigg)\psi(z)\mathrm dz\wedge\ast\big(\psi'(z)\mathrm dz\big) \\ &=\Ree\int_\Sg 20i\psi(z)\overline{\psi'(z)}\frac{\big(1+x^4+y^4+2x^2(1+y^2)\big)^2}{y^2}\mathrm dz\wedge\mathrm d\bar z \\ &=80\g_{\text{WP}}(\psi,\psi')+ \ \text{other therms} \ .
    \end{align*}As it can be seen, there is indeed a part in $\g|_{\mathcal F(\Sg)}$ that coincides with a multiple of Weil-Petersson metric.
\section{Inclusions for $n\ge 3$}
In this final part, making use of the theory of maximal polystable $\SO_0(2,n+1)$-Higgs bundles, we want to understand whether representations $\rho\in\charvarietynmaximal$ whose Zariski closure is contained in $\SO_0(2,3)$ represent smooth or orbifold points in the maximal $\SO_0(2,n+1)$-character variety. In particular, in Section \ref{sec:5.1} we focus on representations whose first Stiefel-Whitney class is non-zero. Then, in Section \ref{sec:4.4}, we analyze the properties of the metric on the equivalent of the so-called Gothen components for $\mathbb P\mathrm{Sp}(4,\R)$ (\cite{gothen2001components}). We explain in which connected components of the maximal $\SO_0(2,n+1)$-character variety they can be deformed to each other, and we perform a similar study as the aforementioned case. The same strategy applies to the Hitchin component.
\subsection{The case $sw_1\neq 0$}\label{sec:5.1}
Let us consider the isometric embedding $\R^{2,3}\to\R^{2,n+1}$ which sends the point $(x_1,x_2,x_3,x_4,x_5)$ to $(x_1,x_2,x_3,x_4,x_5,0,\dots,0)$. At the level of Lie groups we get a tightly embedded copy of $\SO_0(2,3)$ inside $\SO_0(2,n+1)$ in the same way as explained in Section \ref{sec:4.2} for a similar case. In particular, given a maximal representation $\rho\in\charvarietynmaximal^{sw_1\neq 0}_{sw_2}$ such that $\overline{\rho\big(\pi_1(\Sg)\big)}<\SO_0(2,3)\times\SO(1)\times\dots\times\SO(1)$, it defines a point in $\charvarietymaximalstiefel$ (see decomposition (\ref{eq:decompositionSO23})). In other words, we have an inclusion $\varsigma:\charvarietymaximalstiefel\hookrightarrow\charvarietynmaximal^{sw_1\neq 0}_{sw_2}$ at the level of connected components.  \begin{proposition}\label{prop:orbifoldpointsforngeneral}
Any $\rho\in\charvarietymaximalstiefel$ represents a smooth point in $\charvarietynmaximal^{sw_1\neq 0}_{sw_2}$ if $n=3$ and an orbifold point of type $\mathcal C_\rho/Z_{n+3}^\mathbb C$ when $n\ge 4$, where $\mathcal{C_\rho}$ is the centralizer of $\rho\big(\pi_1(\Sg)\big)$ inside $\SO_0(2,n+1)$ and $Z_{n+3}^\mathbb C$ is the center of $\SO(n+3,\mathbb C)$.
\end{proposition}
\begin{proof}
First of all notice that if $\rho:\pi_1(\Sg)\to\SO_0(2,n+1)$ is maximal and its Zariski closure is contained in $\SO_0(2,3)\times\SO(1)\times\dots\times\SO(1)$, then there is a reduction of the structure group of the associated Higgs bundle from $\SO_0(2,n+1)$ to $\SO_0(2,3)$. In particular, if $(\mathcal W,b_{\mathcal W},q_2,\beta_0)$ is the maximal $\SO_0(2,n+1)$-Higgs associated with $\rho$, then $\mathcal W=\mathcal W'\oplus\mathcal O_X^{\oplus^{n-2}}$ where the quadr-uple $(\mathcal W',b_{\mathcal W'},q_2,\beta_0)$ defines a maximal polystable $\SO_0(2,3)$-Higgs bundle whose first Stiefel-Whitney class is non zero. For this reason, since the associated $\SO(5,\C)$ Higgs bundle is stable (\cite[Proposition 4.16]{alessandrini2019geometry}), the $\SO(n+3,\mathbb C)$-bundle associated with $(\mathcal W,b_{\mathcal W},q_2,\beta_0)$ is stable as well (indeed the two differ by the sum of a trivial holomorphic bundle). In particular, the centralizer $\mathcal C_\rho<\SO_0(2,n+1)$ has to be finite and the type of orbifold singularity is detected by $\mathcal C_\rho/Z_{n+3}^\mathbb C$ (\cite[Proposition 2.12]{alessandrini2019geometry}). In the end, it is easy to see that the only value of $n$ for which $\rho$ can be a smooth point is $n=3$, as in this case $Z_6^\mathbb C=\{\pm\mathrm{Id}_6\}$ and those are the only possible matrices in $\SO_0(2,4)$ fixing $\rho\big(\pi_1(\Sg)\big)<\SO_0(2,3)\times\SO(1)$, i.e. $\mathcal C_\rho=Z_6^\mathbb C$. As soon as $n>3$, the centralizer $\mathcal C_\rho$ is always strictly bigger than the center of $\SO(n+3,\mathbb C)$.
\end{proof}
    \begin{lemma}\label{lem:gcompatibleorbifoldforanyn}
        The Riemannian metric $\g$ on $\charvarietynmaximal^{sw_1\neq 0}_{sw_2}$ is compatible with all orbifold singularities arising from representations contained in $\charvarietymaximalstiefel$. 
    \end{lemma}
   \begin{proof}
   Let $\rho\in\charvarietymaximalstiefel\hookrightarrow\charvarietynmaximal^{sw_1\neq 0}_{sw_2}$ with $n\ge 4$. Let $L$ be any matrix in the quotient group $\mathcal{C}_\rho/Z_{n+3}^\mathbb C$, which is not trivial according to Proposition \ref{prop:orbifoldpointsforngeneral}. If $g$ denotes the Riemannian metric at the level of $\lso_0(2,n+1)_{\Ad\rho}$-valued $1$-forms, then with the same approach as in Lemma \ref{lem:invariantmetricg} we get $$g\big(\sigma\otimes L\phi L^{-1},\sigma'\otimes L\phi'L^{-1}\big)=g\big(\sigma\otimes\phi,\sigma'\otimes\phi'\big) \ ,$$ for any $\sigma\otimes\phi\in\Omega^1\big(\Sg,\lso_0(2,n+1)_{\Ad\rho}\big)$. The next step is to prove that the action of any $L$ by conjugation preserves the harmonic $\lso_0(2,n+1)_{\Ad\rho}$-valued $1$-forms. Recall that $\alpha\in\Omega^1\big(\Sg,\lso_0(2,n+1)_{\Ad\rho}\big)$ is harmonic if and only if $\mathrm d\alpha=\delta\alpha=0$ (see Section \ref{sec:3.1}). Let $\sum_i\sigma_i\otimes\phi_i$ be a harmonic representative in its cohomology class, then we need to show that $\mathrm d\big(\sum_i\sigma_i\otimes L\phi_iL^{-1}\big)=0$ and $\delta\big(\sum_i\sigma_i\otimes L\phi_iL^{-1}\big)=0$, for any $L\in\mathcal C_\rho/Z_{n+3}^\mathbb C$. The claim about the differential $\mathrm d$ follows by linearity. It only remains to prove that $$\mathrm d\ast\big(\sum_i\sigma_i\otimes\#(L\phi_iL^{-1})\big)=0 ,$$ as also occurred in Theorem \ref{thm:metriccompatiblewithorbifold}. At this point, since any matrix $M$ in $\lso_0(2,n+1)$ can be written as $\begin{psmallmatrix}
       A & B \\ B^t & D \end{psmallmatrix}$, where $A$ and $D$ are respectively a $2\times 2$ and $(n+1)\times(n+1)$ anti-symmetric matrix and $B$ is a $2\times(n+1)$ arbitrary matrix, we can pick a basis $\{E_1,E_2,E_3,\dots,E_{r}\}$ of $\lso_0(2,n+1)$, where $r=\mathrm{dim} \ \lso_0(2,n+1)$, as \begin{align*}
     E_1=\begin{psmallmatrix}
        0 & -1 & 1 & 0 & 0 \\ 1 & 0 & 0 & 0 & 0 \\ 1 & 0 & 0 & 0 & 0 \\ 0 & 0 & 0 & 0 & 0 \\ 0 & 0 & 0 & 0 & 0 
    \end{psmallmatrix},\quad E_2=\begin{psmallmatrix}
        0 & 0 & 0 & 0 & 0 \\ 0 & 0 & 1 & 0 & 0 \\ 0 & 1 & 0 & 0 & 0 \\ 0 & 0 & 0 & 0 & 0 \\ 0 & 0 & 0 & 0 & 0
    \end{psmallmatrix},\quad E_3=\begin{psmallmatrix}
        0 & 1 & 1 & 0 & 0 \\ -1 & 0 & 0 & 0 & 0 \\ 1 & 0 & 0 & 0 & 0 \\ 0 & 0 & 0 & 0 & 0 \\ 0 & 0 & 0 & 0 & 
    \end{psmallmatrix} \ ,\end{align*}and the remaining $r-3$ matrices are chosen following the same order as the proof of Theorem \ref{thm:metriccompatiblewithorbifold} generalized to arbitrary dimension. Also, another observation is that every matrix $L\in\mathcal C_\rho/Z_{n+3}^\mathbb C$ can be written as a diagonal matrix with only $\pm 1$. This follows from the fact that every such $L$ has to belong in $\SO_0(2,n+1)$ and has to satisfy the relation $$L\begin{pmatrix}
        C & 0 \\ 0 & \mathrm{Id_{n-2}}
    \end{pmatrix}L^{-1}=\begin{pmatrix}
        C & 0 \\ 0 & \mathrm{Id_{n-2}}
    \end{pmatrix}, \quad \text{for any} \ C\in\SO_0(2,3) \ .$$
    In the end, using the explicit expression of each matrix $L$ and $L=L^{-1}$, we obtain that its action by conjugation on the basis $\{E_1,E_2,E_3,\dots,E_r\}$ leaves it unchanged up to sign. This allows us to follow the same calculations made in the proof of Theorem \ref{thm:metriccompatiblewithorbifold} and conclude that $$\mathrm d\ast\big(\sum_i\sigma_i\otimes\#(L\phi_iL^{-1})\big)=0 \ .$$
   \end{proof}Let us see with two specific examples, what matrices in $\mathcal C_\rho/Z_{n+3}^\mathbb C$ represent orbifold points. As we can deduce from Proposition \ref{prop:orbifoldpointsforngeneral}, the first interesting case to analyze is for $n=4$. For instance, we need to understand how all matrices $L$ in $\SO_0(2,5)$ such that $$L\begin{pmatrix}
        C & 0 \\ 0 & \mathrm{Id_{2}}
    \end{pmatrix}L^{-1}=\begin{pmatrix}
        C & 0 \\ 0 & \mathrm{Id_{2}}
    \end{pmatrix}, \quad \text{for any} \ C\in\SO_0(2,3),$$ look like. It is not hard to see that they are given by $\big\{\mathrm{Id}_7, \big(\mathrm{Id}_5,-\mathrm{Id}_2\big), \big(-\mathrm{Id}_6,1\big), \big(-\mathrm{Id}_5,\mathrm{diag}(1,-1)\big)\big\}$. Since $Z_7^\mathbb C$ is trivial, we get that $$\mathcal{C}_\rho/Z_7^\mathbb C\cong\Z_2\times\Z_2=\big\{\mathrm{Id}_7, \big(\mathrm{Id}_5,-\mathrm{Id}_2\big), \big(-\mathrm{Id}_6,1\big), \big(-\mathrm{Id}_5,\mathrm{diag}(1,-1)\big)\big\} \ .$$ As for the subsequent case, namely $n=5$, we have that $Z_8^\mathbb C=\{\pm\mathrm{Id}_8\}$. Then, with the analogous argument as above we get \begin{align*}
        \mathcal{C}_\rho=\big\{\pm\mathrm{Id}_8, \pm\big(-\mathrm{Id}_6,\mathrm{Id}_2\big), \pm\big(-\mathrm{Id}_5,\diag(1,-1,1)\big), \pm\big(-\mathrm{Id}_5,\diag(1,1,-1)\big)\big\} \ ,
    \end{align*} so that the quotient by $Z_8^\mathbb C$ is identified with an isomorphic copy of $\Z_2\times\Z_2$.
In the end, iterating the procedure for arbitrary $n\ge 4$, we obtain  $$\bigslant{\mathcal C_\rho}{Z_{n+3}^\mathbb C}\cong\bigslant{\big(\Z_2\big)^{\times^{n-2}}}{Z_{n+3}^\mathbb C}, \ \text{where} \ Z_{n+3}^\mathbb C=\begin{cases}
    \big\{\mathrm{Id}_{n+3}\big\} \ \ \text{if $n$ is even} \\  \big\{\pm\mathrm{Id}_{n+3}\big\} \ \ \text{if $n$ is odd}
\end{cases}$$

\begin{theorem}\label{thm:totallygeodesicngeneral}
For any $n\ge 3$, the space $\big(\charvarietymaximalstiefel,\g\big)$ is totally geodesic in $\big(\charvarietynmaximal^{sw_1\neq 0}_{sw_2},\g\big)$.
\end{theorem}
\begin{proof}
The proof follows the lines of the argument used in Section \ref{sec:4.3}. Let us first consider the map $q:\charvarietynmaximal^{sw_1\neq 0}_{sw_2}\longrightarrow\charvarietynmaximal^{sw_1\neq 0}_{sw_2}$ which sends the representation $\rho$ to $Q\rho Q^{-1}$, where $Q:=\diag(-\mathrm{Id}_5,\mathrm{Id}_{n-2})=Q^{-1}\in\mathrm{O}(2,n+1)$. It is clear that $q$ fixes the representations contained in $\charvarietymaximalstiefel$. Hence, we need to prove that the map $q_\ast$ induced on $\lso_0(2,n+1)$-valued $1$-forms preserves the Riemannian metric $g$ defined in (\ref{eq:definitiong}) and sends harmonic forms to harmonic forms. The former follows directly by extending the proof of Lemma \ref{lem:invariantmetricgtotallygeodesic} to an arbitrary $n\ge 3$, indeed the existence of the associated $\rho$-equivariant maximal space-like surface in $\Hyp$ is guaranteed for any $n\ge 2$ (Theorem \ref{thm:rhoequivariantsurface}). The latter is a combination of the computations performed in Proposition \ref{prop:harmonicformsantidesitter} together with the argument used in Lemma \ref{lem:gcompatibleorbifoldforanyn}. As a result, we conclude that the map $q$ is an isometry with respect to $\g$ whose fixed points locus is the space $\charvarietymaximalstiefel$ which is therefore totally geodesic.
\end{proof}
\subsection{Gothen and Hitchin components}\label{sec:4.4}
Recall that there are $4g-3$ connected components contained in $\charvarietymaximal$, with $sw_1=0$, parameterized by an integer $d\in[0,4g-4]$ that corresponds to the degree of a certain holomorphic bundle over the surface, associated with the representation (see Section \ref{sec:4.1}). The ones corresponding to $d\in(0,4g-4)$ are the so-called Gothen components and denoted with $\mathcal{M}_d(\Sg)$ (see Section \ref{sec:4.1}), on the other hand when $d=4g-4$ we retrieve the Hitchin component $\Hit(\Sg)$. Thus, for any $d\in(0,4g-4]$, the space $\mathcal{M}_d(\Sg)$ is smooth, hence it carries a well-defined Riemannian metric $\g$ (Theorem \ref{thm:metricGothencomponentHitchin}). As soon as $n\ge 3$, there are no more such exceptional components in the maximal character variety (Theorem \ref{thm:decompositioncomponents}), and it is therefore not clear a-priori how the inclusion induced by the isometric embedding \begin{align*}& \ \ \ \ \ \ \ \ \ \quad\R^{2,3}\longrightarrow\R^{2,n+1}  
\\ &(x_1,x_2,x_3,x_4,x_5)\mapsto(x_1,x_2,x_3,x_4,x_5,0,\dots,0)\end{align*} acts on representations. Actually, this problem has already been studied by Collier (\cite[Proposition 5.6]{collier2018psl}) using Higgs bundle theory. Here we will give some details on the aforementioned inclusion, and then try to figure out what kind of singularities Gothen and Hitchin components form when $n\ge 3$. \\ \\
In Section \ref{sec:4.1} we saw that if $(\mathcal W,b_\mathcal W,q_2,\beta_0)$ is a maximal polystable $\SO_0(2,3)$-Higgs bundle over $X=(\Sg,J)$, and if the first Stiefel-Whitney class $sw_1$ of $\mathcal W$ vanishes, then $\mathcal W$ is endowed with a $\SO(2,\C)$-structure and there is a further holomorphic splitting $$(\mathcal W,b_\mathcal W)=\bigg(\mathcal F\oplus\mathcal F^{-1},\begin{pmatrix}
    0 & 1 \\ 1 & 0 
\end{pmatrix}\bigg) \ ,$$ where $\mathcal F$ is a holomorphic line bundle over $X$. The degree $d$ of $\mathcal F$ is exactly the topological invariant which distinguishes connected components in $\charvarietymaximal$ with $sw_1=0$. \begin{lemma}\label{lem:gothencomponentinclusion}
Let $n\ge3$ and let $\varsigma:\charvarietymaximal\to\mathfrak{R}^{\text{max}}_{2,n+1}(\Sg)$ be the inclusion induced by the isometric embedding $\R^{2,3}\to\R^{2,n+1}$ described above. Then, for any $d\in[0,4g-4]$, the connected component $\mathcal{M}_d(\Sg)$ is contained in $\mathfrak{R}^{\text{max}}_{2,n+1}(\Sg)^{sw_1=0}_{sw_2=0}$ when $d$ is even and is contained in $\mathfrak{R}^{\text{max}}_{2,n+1}(\Sg)^{sw_1=0}_{sw_2\neq0}$ when $d$ is odd.
\end{lemma}
\begin{proof}
At the level of Higgs bundles, the inclusion described above maps the maximal polystable $\SO_0(2,3)$-Higgs bundle $(\mathcal W',b_{\mathcal W'},q_2,\beta_0)$ to the maximal polystable $\SO_0(2,n+1)$-Higgs bundles determined by the quadr-uple $(\mathcal W,b_{\mathcal W},q_2,\beta_0)$, with $\mathcal W:=\mathcal W'\oplus\mathcal O_X^{\oplus^{n-2}}$. Given that the topological invariants of $\mathcal W'$ are the same as $\mathcal W$, we get the following two refined inclusions \begin{equation}\label{eq:inclusioncharacter}\charvarietymaximalstiefel\hookrightarrow\mathfrak{R}^{\text{max}}_{2,n+1}(\Sg)^{sw_1\neq 0}_{sw_2} \ ,\quad \bigsqcup_{0\le d\le 4g-4}\mathcal{M}_d(\Sg)\hookrightarrow\mathfrak{R}^{\text{max}}_{2,n+1}(\Sg)^{sw_1=0}_{sw_2} \ .\end{equation} In particular, if $(\mathcal W',b_{\mathcal W'},q_2,\beta_0)$ has $sw_1(\mathcal W')=0$ and belongs to one of the Gothen components, the degree $d=\deg(\mathcal F')$ refines the second Stiefel-Whitney class of $\mathcal W'$. In other words, $$sw_2(\mathcal W)=sw_2(\mathcal W')= d \ (\text{mod} \ 2) \ .$$ The claim is obtained by the above observation together with the inclusions (\ref{eq:inclusioncharacter}).
\end{proof}
\begin{proposition}\label{prop:Gothencomponentorbifold}
Let $d\in (0,4g-4]$, then any representation belonging to $\mathcal{M}_d(\Sg)$ inside $\mathfrak{R}^{\text{max}}_{2,n+1}(\Sg)^{sw_1=0}_{sw_2}$ is a smooth point if $n=3$ and it is an orbifold point of type $\mathcal C_\rho/Z_{n+3}^\mathbb C$ if $n\ge 4$, where $\mathcal C_\rho$ is the centralizer of $\rho\big(\pi_1(\Sg)\big)$ inside $\SO_0(2,n+1)$ and $Z_{n+3}^\mathbb C$ is the center of $\SO(n+3,\mathbb C)$. Moreover, for any $n\ge 4$, we have $$\bigslant{\mathcal C_\rho}{Z_{n+3}^\mathbb C}\cong\bigslant{\big(\Z_2\big)^{\times^{n-2}}}{Z_{n+3}^\mathbb C} \ .$$
\end{proposition}
\begin{proof}
Let $\rho\in\mathcal{M}_d(\Sg)\subset\mathfrak{R}^{\text{max}}_{2,n+1}(\Sg)^{sw_1=0}_{sw_2}$, for $d\in(0,4g-4]$. As we already explained (see Section \ref{sec:4.2}), there is a reduction of the structure group of the associated Higgs bundle from $\SO_0(2,n+1)$ to $\SO_0(2,3)$. Using the same notation as the previous lemma, since $sw_1=0$, the quadr-uple associated with $\rho$ is given by $(\mathcal W,b_\mathcal W,q_2,\beta_0)$ where $\mathcal{W}\cong\mathcal W'\oplus\mathcal O_X^{\oplus^{n-2}}, \ \mathcal L^{-1}\otimes K_X\cong\det\mathcal W\cong\mathcal O_X$ and $(\mathcal W',b_{\mathcal W'})=\big(\mathcal F\oplus\mathcal F^{-1},\begin{psmallmatrix}
    0 & 1 \\ 1 & 0 
\end{psmallmatrix}\big)$. The holomorphic splitting for $\mathcal W'$ induces a further decomposition $$\beta_0=\begin{pmatrix}
     \nu  \\
      \mu
\end{pmatrix}:K_X^{-1}\longrightarrow\big(\mathcal F\oplus\mathcal F^{-1}\big)\otimes K_X \ ,$$ where $\nu\in H^0\big(X,\mathcal F\otimes K_X^2\big)$ and $0\neq\mu\in H^0\big(X,\mathcal F^{-1}\otimes K_X^2\big)$. Given that the $\SO(5,\C)$-bundle associated with $(\mathcal F,q_2,\nu,\mu)$ is stable (\cite[Proposition 4.13]{alessandrini2019geometry}), the $\SO(n+3,\C)$-bundle associated with $(\mathcal W,b_\mathcal W,q_2,\beta_0)$ is stable as well (indeed the two differ by the sum of a trivial holomorphic bundle). In particular, the centralizer $\mathcal C_\rho:=C\big(\rho(\pi_1(\Sg))\big)<\SO_0(2,n+1)$ is finite and it is a copy of $\big(\mathbb Z_2\big)^{n-2}$ generated by diagonal matrices in $\SO_0(2,n+1)$ with only $\pm 1$, as we explained in Section \ref{sec:5.1}. Using the same argument as Proposition \ref{prop:orbifoldpointsforngeneral}, we conclude that $\rho$ is smooth point if $n=3$ and an orbifold point if $n\ge 4$.
\end{proof}
\begin{theorem}
For any $d\in (0,4g-4]$, the connected component $\big(\mathcal{M}_d(\Sg),\g\big)$ is totally geodesic in $\big(\mathfrak{R}^{\text{max}}_{2,n+1}(\Sg)^{sw_1=0}_{sw_2=0},\g\big)$ when $d$ is even and is totally geodesic in $\big(\mathfrak{R}^{\text{max}}_{2,n+1}(\Sg)^{sw_1=0}_{sw_2\neq0},\g\big)$ when $d$ is odd.
\end{theorem}
\begin{proof}
The claim follows from a combination of the results obtained above and the strategy used in Section \ref{sec:5.1}. In fact, again, the Zariski closure of the representations we are considering is contained in $\SO_0(2,3)\times\SO(1)\times\dots\times\SO(1)$. Thus, using the argument of Lemma \ref{lem:gcompatibleorbifoldforanyn} we obtain that the Riemannian metric $\g$ on $\charvarietynmaximal^{sw_1}_{sw_2}$ is compatible with the orbifold singularities arising from Proposition \ref{prop:Gothencomponentorbifold}, and using the same as Theorem \ref{thm:totallygeodesicngeneral} that the spaces $\mathcal{M}_d(\Sg)$ are totally geodesic in $\charvarietynmaximal^{sw_1=0}_{sw_2}$ or $\charvarietynmaximal^{sw_1=0}_{sw_2=0}$ with respect to $\g$ and according to the parity of the integer $d$.
\end{proof}

\printbibliography
\end{document}